\documentclass[pdftex,pdfversion=1.4,a4paper,openright,twoside,11pt]{article}
\usepackage[a4paper,left=3.5cm,right=3cm]{geometry}

\usepackage[T1]{fontenc}
\usepackage[english]{babel}
\usepackage[utf8]{inputenc}
\usepackage{amsmath}
\usepackage{wasysym}
\usepackage{amsfonts}
\usepackage{amssymb}
\usepackage{amsthm}
\usepackage{comment}
\usepackage{authblk}
\usepackage{mathtools}
\usepackage{xcolor}
\usepackage[shortlabels]{enumitem}
\usepackage{graphicx}
\usepackage{dsfont}
\usepackage{textcomp}

\theoremstyle{plain}
\newtheorem{theorem}{Theorem}[section]
\newtheorem{prop}[theorem]{Proposition}
\newtheorem{lemma}[theorem]{Lemma}
\newtheorem{cor}[theorem]{Corollary}
\newtheorem*{theorem*}{Theorem}
\newtheorem*{prop*}{Proposition}
\newtheorem*{lemma*}{Lemma}
\newtheorem*{cor*}{Corollary}

\graphicspath{{./Immagini/}}

\theoremstyle{definition}
\newtheorem{defi}[theorem]{Definition}
\newtheorem*{defi*}{Definition}

\theoremstyle{remark}
\newtheorem{oss}[theorem]{Remark}
\newtheorem*{oss*}{Remark}

\theoremstyle{remark}

\newcommand{\norm}[1]{\left\lVert#1\right\rVert}

\newcommand{\doubleR}{\mathds{R}}
\newcommand{\scriptC}{\mathcal{C}}
\newcommand{\epsi}{\varepsilon}
\newcommand{\scriptB}{\mathcal{B}}
\newcommand{\doubleC}{\mathds{C}}

\newcommand{\scriptG}{\mathcal{G}}

\newcommand{\scriptA}{\mathcal{A}}

\newcommand{\scriptM}{\mathcal{M}}

\DeclareMathOperator{\sgn}{sgn}

\DeclareMathOperator{\diag}{diag}

\DeclareMathOperator{\Image}{Im}

\DeclareMathOperator{\agInt}{a-int}

\title{On the determination of \textcolor{black}{Lagrange multipliers for a weighted} LASSO problem using geometric and convex analysis techniques}
\date{}
\author[1, 2, 3, 4]{Gianluca Giacchi \thanks{gianluca.giacchi2@unibo.it}}
\author[3]{Bastien Milani \thanks{bastien.milani@chuv.ch}}
\author[2,4]{Benedetta Franceschiello \thanks{benedetta.franceschiello@hevs.ch}}
\affil[1]{Università di Bologna, Dipartimento di Matematica, Piazza di Porta San Donato 5, 40126 Bologna, Italy}
\affil[2]{Institute of Systems Engineering, School of Engineering, HES-SO Valais-Wallis, Rue de l'Industrie 23, 1950 Sion, Switzerland}
\affil[3]{Lausanne University Hospital and University of Lausanne, Lausanne, Department of Diagnostic and Interventional Radiology, Rue du Bugnon 46, Lausanne 1011, Switzerland}
\affil[4]{The Sense Innovation and Research Center, Avenue de Provence 82 1007, Lausanne and Ch. de l’Agasse 5, 1950 Sion, Switzerland}

\begin{document}
\maketitle

\begin{abstract}
\textcolor{black}{Compressed Sensing (CS) encompasses a broad array of theoretical and applied techniques for recovering signals, given partial knowledge of their coefficients, cf. \cite{Candes, CRT, Donoho, DET}. Its applications span various fields, including mathematics, physics, engineering, and several medical sciences, cf. \cite{AH, Berk2, Brady, Chan, Arce, Gao, Liu, MW, Marim, Wang, Zhu}. Motivated by our interest in the mathematics behind Magnetic Resonance Imaging (MRI) and CS, we employ convex analysis techniques to analytically determine equivalents of Lagrange multipliers for optimization problems with inequality constraints, specifically a weighted LASSO with voxel-wise weighting. We investigate this problem under assumptions on the fidelity term $\norm{Ax-b}_2^2$, either concerning the sign of its gradient or orthogonality-like conditions of its matrix. To be more precise, we either require the sign of each coordinate of $2(Ax-b)^TA$ to be fixed within a rectangular neighborhood of the origin, with the side lengths of the rectangle dependent on the constraints, or we assume $A^TA$ to be diagonal. The objective of this work is to explore the relationship between Lagrange multipliers and the constraints of a weighted variant of LASSO, specifically in the mentioned cases where this relationship can be computed explicitly. As they scale the regularization terms of the weighted LASSO, Lagrange multipliers serve as tuning parameters for the weighted LASSO, prompting the question of their potential effective use as tuning parameters in applications like MR image reconstruction and denoising. This work represents an initial step in this direction.}
\end{abstract}


%\tableofcontents

\section{Introduction}
\textcolor{black}{Basis Pursuit is a well-known convex minimization problem that was first introduced by F. Santosa and W. W. Symes in 1986, cf. \cite{SS}, in its simplest formulation:
\begin{equation}\label{lassetto1}
	\text{minimize} \quad \norm{x}_1+\lambda\norm{Ax-b}_2^2, \qquad x\in\mathds{R}^{n},
\end{equation}
where $A\in\mathds{R}^{m\times n}$, the so-called \textit{design matrix}, and $b\in\mathds{R}^m$ are fixed. The same problem was later applied to signal processing by S. S. Chen and D. Donoho in 1994, cf. \cite{CD}. In 1996, R. Tibshirani re-introduced it as linear regression method, under the name of LASSO. Namely, in \cite{T}, they consider the constrained minimization problem:
\begin{equation}\label{Tiblasso}
	\text{minimize}\quad \norm{Ax-b}_2^2, \qquad x\in\doubleR^n, \ \norm{x}_1\leq\tau,
\end{equation}
for $\tau>0$ and
\[
    \norm{x}_1:=\sum_{j=1}^n|x_j|.
\]
We will discuss the equivalence between (\ref{lassetto1}) and (\ref{Tiblasso}) in the following.\\
}

\textcolor{black}{Mathematical analysis approaches to study LASSO problems in all their facets are not new, and the literature is so vast that we can only limit ourselves to mention a few examples. The interested reader can find more in \cite{Berk, CDS, DVFP, MY, OH, OTH, SBA}. In \cite{Unser1, Unser2}, the authors study representation theorems for the solutions of general problems:
\[
	\arg\min_{x}E(b,\nu(x))+\gamma(\norm{x}),
\]
in the framework of Banach space theory, where $E$ is a loss functional, $\nu$ is a measurement mapping, $\gamma$ is a strictly increasing convex function and $\norm{\cdot}$ is a Banach norm, we refer to \cite[Theorem 2]{Unser1}, \cite[Theorem 2, Theorem 3]{Unser2} for more precise statements.
}
\textcolor{black}{
In \cite{Berk3}, the authors use convex analysis and variational calculus to study regularity properties of the mapping:
\[
	(b,\lambda)\in\mathds{R}^m\times(0,+\infty)\mapsto \arg\min_{x\in\mathds{R}^n}\frac{1}{2}\norm{Ax-b}_2^2+\lambda\norm{x}_1.
\]
}

\textcolor{black}{As aforementioned}, in its simplest definition, LASSO consists in the minimization of the function:
\begin{equation}\label{lassetto}
	%\arg\min_{x\in\mathds{R}^n}
	\norm{Ax-b}_2^2+\lambda\norm{x}_1,
\end{equation}
where $A\in\doubleR^{m\times n}$ and $b\in\doubleR^m$ is a measurement vector. \textcolor{black}{Clearly, problems (\ref{lassetto1}) and (\ref{lassetto}) have the same minimizers and, therefore, for the purposes of this work, we will consider them as the same minimization problem. In short, (\ref{lassetto1}) and (\ref{lassetto}) can be interpreted as regularization problems,} where the aim is to minimize simultaneously the fidelity term $\norm{Ax-b}_2^2$, that measures noise, and the regularization term $\norm{ x}_1$, that enforces sparsity. \textcolor{black}{Recall that a vector $x=(x_1,\ldots,x_n)$ is $s-$\textbf{sparse} if $\text{card}\{j : x_j\neq0\}\leq s$. When $s$ is clear from the context or irrelevant, we drop $s$ and say that $x$ is \textit{sparse}. In several applications, $x$ is not sparse itself, but it is sparse with respect to a so-called \textit{sparsity-promoting transform} $\Phi:\doubleR^n\to\doubleR^N$. Stated differently, when $\Phi x$ is known to be sparse, problem (\ref{lassetto}) can be generalized to:
\begin{equation}\label{gLASSO}
    \text{minimize}\quad \norm{Ax-b}_2^2+\lambda\norm{\Phi x}_1, \qquad x\in\doubleR^n, 
\end{equation}
i.e. the regularization term $\norm{x}_1$ in (\ref{lassetto}) is replaced by $\norm{\Phi x}_1$.}
The parameter $\lambda>0$ in (\ref{lassetto}) acts as a tuning parameter that balances the contributions of the {fidelity} term $\norm{Ax-b}_2^2$ and regularization addendum $\norm{\Phi x}_1$: small values of $\lambda$ lower the contribution of the regularization, strengthening the effect of the fidelity term; vice-versa, large values of $\lambda$ make $\norm{Ax-b}_2^2$ negligible and force $\norm{\Phi x}_1$ to be small in order for the overall sum to be small. Consequently, solutions corresponding to $\lambda\ll1$ will be noisy, being close to the set $A^{-1}b$, while solutions $x^\#$ corresponding to $\lambda\gg1$ have more sparse $\Phi x^\#$. 
From this perspective, estimates of tuning parameters for inverse problems can be performed pursuing  different approaches. \textcolor{black}{ \textit{A posteriori rules} can be used when some a-priori knowledge on the amplitude of noise $e\in\doubleR^m$ is available, say $\norm{e}_2\leq\epsi$. For instance, using Morozov's discrepancy principle, $\lambda$ can be chosen so that a solution $x_\lambda$ of (\ref{lassetto}) satisfies $\norm{Ax_\lambda-b}_2\leq\epsi$, cf. \cite{CLPS, HHZS, LLL}. \textit{A priori rules} require knowledge of noise level, as before, but also a-priori information on the regularity of the solution. For this reason, a-priori approaches are usually bad suited for applications, cf. \cite{AMOS}.
Heuristic methods, such as the L-curve are also available, cf. \cite{CMRS, Hansen, LKB}. The L-curve method consists of choosing the optimal tuning parameter empirically by tracing a trade-off curve (the L-curve), whereas the generalized cross-validation (GCV) is a well-performing method that requires high-dimensional matrix calculus, cf. \cite{KNN, VMBM, GHW}. Other non-standard methods can be found in \cite{LPS, PCLS}, where the parameter is chosen so that statistical properties of noise, such as whiteness, are optimized; an implementation that avoids the computation of matrix inverses can be found in \cite{BKBVGS}. CNN and other learning methods were deployed in \cite{HP, NDC}, while a more statistical point of view was adopted in \cite{CLPS2}.} \\

However, the very reason why $\lambda$ is interpreted as a trade-off between noise and sparsity, in (\ref{lassetto1}) and (\ref{lassetto}), is that it depends on estimates that are usually unavailable, such as a-priori upper bounds for the $\ell_1$ norm of the unknown \textcolor{black}{vector}, i.e. a-priori information on the sparsity of the solution, or upper bounds for the noise, cf. \cite{OTH}. %This dependence can be highlighted working on the expression of (\ref{lassetto1}). 
\textcolor{black}{For $A\in\doubleR^{m\times n}$, $b\in\doubleR^m$ and $\eta\geq0$,} the function:
\begin{equation}\label{lagrangian}
	L(x,\lambda)=\norm{x}_1+\lambda(\norm{Ax-b}_2^2-\eta^2)
\end{equation}
is the Lagrangian associated to the constrained minimization problem:
\begin{equation}\label{lassettino}
	\text{minimize} \quad \norm{x}_1, \qquad \text{$x\in\doubleR^n$, $\norm{Ax-b}_2^2\leq\eta^2$},
\end{equation}
cf. \cite{BV}. \textcolor{black}{Roughly speaking, this entails that} (\ref{lassetto1}) and (\ref{lassettino}) are equivalent, up to choosing: 
\begin{equation}\label{ltau}
	\lambda=\lambda(\eta)
\end{equation}
or, equivalently, $\eta=\eta(\lambda)$, properly. Please note that $\eta$ may not be uniquely determined. \textcolor{black}{We refer to \cite[Proposition 3.2]{FR} for a more precise statement of this fact. Throughout this work, we call the parameter $\lambda$ in (\ref{ltau}) a \textit{Lagrange multiplier} associated to (\ref{lassettino}), since it plays the same role of Lagrange multipliers in optimization problems with equality constraints. We will use this terminology in a more general setting, see Section Definition \ref{defLOP} below.} \textcolor{black}{Since a slightly modified proof of \cite[Theorem 3.1]{FR} shows that a solution of (\ref{lassettino}), if unique, must be $m$-sparse, $\norm{\cdot}_1$ is said to \textit{enforce sparsity}. For this reason, the Lagrange multiplier in (\ref{lagrangian}) could be used, in principle, in an equivalent manner as a \textit{tuning parameter} for (\ref{lassetto1}) to recover sparse vectors}. \\

\textcolor{black}{In the same way,
\begin{equation}\label{lagrange2}
	L(x,\lambda)=\norm{Ax-b}_2^2+\lambda(\norm{\Phi x}_1-\tau)
\end{equation}
is the Lagrange function of the constrained problem:
\begin{equation}\label{lagrange23}
	\text{minimize} \qquad \norm{Ax-b}_2^2, \qquad \text{$x\in\doubleR^n$, $\norm{\Phi x}_1\leq\tau$}.
\end{equation}
A first question that may be addressed is whether the corresponding Lagrange multiplier $\lambda$ of (\ref{lagrange23}) could still be used as a tuning parameter in (\ref{gLASSO}). If so,  
the relationship between the Lagrange multipliers and the constraints of the corresponding constrained problems could be useful in concrete applications, such as static and dynamic MRI, cf. \cite{DsPCSY,FDsMIDZNBJYSM}.  
In MRI, indeed,} vectors {of interest} are MR images, which tend to be approximately sparse with respect to the discrete Fourier transform (DFT), the discrete cosine transform (DCT) or the discrete wavelet transform (DWT), cf. \cite{Lustig}. This means that a solution of the generalized LASSO problem (\ref{gLASSO}), where the {design matrix} $A$ \textcolor{black}{is a proxy of the acquisition methods properties (coil sensitivity, undersampling schemes and DFT), and $b$ is an underdetermined, noisy measurement, will have a sparse regularization term, i.e. sparse $\Phi x$.} \textcolor{black}{We stress that (\ref{gLASSO}) is known to admit a, in general not unique, solution for any choice of $A$, $b$, $\lambda$ and $\Phi$, and for the sake of completeness we report the proof in the appendix.} In order to exploit more a-priori knowledge on the structure of MRI data, \textcolor{black}{(\ref{gLASSO}) can be generalized further to consider target functions that are sum of more regularizing terms, cf. \cite{Yerly, GVV, PV}}. \\

\textcolor{black}{Let us note that sparsity is not always the correct assumption in MRI.} For instance, dynamic MR images (e.g. a sequence of images of a moving organ, cf. \cite{DsPCSY,FDsMIDZNBJYSM}) are highly compressible, rather than sparse, cf. \cite{LDealtri}. \textcolor{black}{This means that most of their coefficients with respect to some sparsity-promoting transform do not vanish, yet are small or negligible}. \\

Surprisingly, it is easier to identify an equivalent of the relationship (\ref{ltau}), between the parameter $\lambda$ and \textcolor{black}{the upper bound for the constraint}, $\eta$, when another \textcolor{black}{weighted} version of LASSO is considered. \textcolor{black}{Namely, we aim to \textcolor{black}{utilize convex analysis to compute the Lagrange multipliers for the constrained optimization} problem:
\begin{equation}\label{constSVbasisP}
        \text{minimize} \quad \norm{Ax-b}_2^2, \qquad x\in\doubleR^n, \quad  |x_j|\leq\tau_j, \quad \text{$j=1,\ldots,n$}.
\end{equation}
For given $\tau_1,\ldots,\tau_n>0$ and a given minimizer $x^\#$ of (\ref{constSVbasisP}) there exist $\lambda_1,\ldots,\lambda_n\geq0$ such that $x^\#$ is also a minimizer of:
\begin{equation}\label{SVbasisP}
    \text{minimize}\quad \norm{Ax-b}_2^2+\sum_{j=1}^n\lambda_j|x_j|,
\end{equation}
see \cite[Section 5.3.2]{BV} or Theorem \ref{thm1} below for a complete statement.} Other \textcolor{black}{weighted} versions of this problem have been considered in the literature. For instance, in \cite{NDC}, the authors present a total variation (TV) regularization-based \textcolor{black}{weighted} LASSO for image denoising. Other references include \textcolor{black}{\cite{CLPS}, where the authors consider \textit{space-variant} problems, such as:
\[
    \text{minimize} \quad \frac{1}{2}\norm{Ax-b}_2^2+\sum_{j=1}^k\lambda_k\norm{(Dx)_j}_p,
\]
where $A\in\doubleR^{m\times n}$, $b\in\doubleR^m$, and $D$ is the discrete gradient, $p\in\{1,2\}$, and $\lambda_1,\ldots,\lambda_k>0$. In a certain sense, problem (\ref{SVbasisP}) can be considered as a space-variant problem, where every component of the unknown vector is weighted by a different parameter. In \cite{PCarxiv}, the author discusses the importance of space-variance in TV regularization, as a mathematical modeling which has the advantage of recovering a description of local features, which is lost by classical TV regularization, i.e. (\ref{gLASSO}) with $\Phi=D$.}\\

As we shall see, the relationship between these parameters \textcolor{black}{is non-trivial if $A$ is non orthogonal}, due to the complicated geometry of (\ref{constSVbasisP}). \textcolor{black}{Loosely speaking, this is due to the fact that if $A^TA$ is non-diagonal, $A$ shuffles the coordinates of $x$ in such a way that each pair of sets $M_j:=\{x\in\doubleR^n : x_j=-\tau_j\}$ and $N_j:=\{x\in\doubleR^n:\frac{\partial}{\partial x_j}(\norm{Ax-b}_2^2)=0\}$ are no longer parallel.}

As for the applicability of these results, the dependence on $\tau_1,\ldots,\tau_n$ of the Lagrange multipliers may turn out to be too restrictive, because these estimates are not generally available. However, we \textcolor{black}{presume} that in certain situations, such as denoising, the solution obtained using the estimated $\tau_j$'s leads to high-quality denoising.\\

We also point out that, unless $\tau_j=0$ for some $j=1,\ldots,n$, a solution \textcolor{black}{$x^\#\in\doubleR^n$ of (\ref{constSVbasisP}) may not have zero entries. For the sake of example, the minimizer of:
\[
    \text{minimize} \quad (x-2)^2+(y-2)^2 \qquad \text{subject to $(x,y)\in\doubleR^2$, $|x|\leq1$, $|y|\leq1$}
\]
is $x^\#=(1,1)$. Clearly, depending on the size of $\tau_j$'s, problem (\ref{constSVbasisP}), and therefore problem (\ref{SVbasisP}), can lead to compressible solutions.}\\

\textbf{Overview.} In Section \ref{sec:pan}, we establish the notation we use in this work. 
In Section \ref{sec:LASSO} we compute the deterministic \textcolor{black}{relationships between the parameters $\lambda_j$'s and the $\tau_j$'s} in order for problems (\ref{constSVbasisP}) and (\ref{SVbasisP}) to be equivalent, \textcolor{black}{under the following specific assumptions: given $A$ such that $A^TA$ is diagonal, for instance when $A$ is either a subsampling matrix, the Fourier transform matrix or the identity matrix, the Lagrange multipliers are explicitly given by:
\begin{equation}\label{opt1}
		\lambda_j^\#=2\norm{a_{\ast,j}}_2^2\Big(\frac{|\langle b,a_{\ast,j}\rangle|}{\norm{a_{\ast,j}}_2^2}-\tau_j\Big)\chi_{\Big[0,\frac{|\langle b,a_{\ast,j}\rangle|}{\norm{a_{\ast,j}}_2^2}\Big]}(\tau_j),
	\end{equation}
where $a_{\ast,j}$ denotes the $j$-th column of $A$ and $\chi_{\Big[0,\frac{|\langle b,a_{\ast,j}\rangle|}{\norm{a_{\ast,j}}_2^2}\Big]}$ is the characteristic function on $\Big[0,\frac{|\langle b,a_{\ast,j}\rangle|}{\norm{a_{\ast,j}}_2^2}\Big]$, $j=1,...,n$. 
We also provide deterministic results for those cases where there is a control on the sign of the gradient of $\norm{Ax-b}_2^2$, providing the explicit expression of the Lagrange multipliers under the assumption $\frac{\partial}{\partial x_j}(\norm{Ax-b}_2^2)\leq0$ for every $j=1,\ldots,n$ in a properly defined hypercube. The conclusions are reported in Section \ref{sec:conclusions}.}

We point out that our result is interesting for two main reasons: to the best of our knowledge, the analytic dependence that we investigate was never fully understood, neither computed. Formula (\ref{opt1}) can be applied directly for denoising in some transform domain, i.e. if $A$ is orthogonal or the identity itself. Furthermore, if $A$ is the matrix of an undersampling pattern, $A^TA$ is diagonal and (\ref{opt1}) can be exploited to control the \textcolor{black}{Lagrange multipliers} of the \textcolor{black}{weighted} LASSO problem (\ref{SVbasisP}) in terms of voxel-wise estimates. We presume that such estimates are relatively easy to obtain. For example, one may first reconstruct a highly undersampled image, apply filters using convolution techniques and estimate the tuning parameters $\lambda$ from the filtered image, therefore obtaining a denoised image. 

\section{Preliminaries and notation}\label{sec:pan}
\textcolor{black}{\paragraph{Notation.} For the theory of this section, we refer to \cite{BV,FR,R3} as reference therein. We denote by $\doubleR^n$ the $n$-dimensional vector space of real column vectors, whereas $\doubleR^{m\times n}$ denotes the space of real $m\times n$ matrices. To ease the notation, if  $x\in\doubleR^n$, the notation $x=(x_1,...,x_n)$ means that $x$ is the column vector with coordinates $x_1,...,x_n$. If $A\in\doubleR^{m\times n}$, $A^T$ denotes the transpose of $A$.}

\textcolor{black}{If $A\in\doubleR^{m\times n}$, $\ker(A)$ and $\Image(A)$ denote the kernel and the image of $A$, respectively. $\scriptM_n$ denotes the set of $n\times n$ signature matrices and, for $x\in \doubleR^n$, $\sgn(x)$ denotes the set of all the possible signatures of $x$, see Section \ref{subsec:33} below.}

\textcolor{black}{For $1\leq p<\infty$, the  $\ell_p$-norm on $\doubleR^n$ is defined as:
\[
\norm{x}_p:=\left(\sum_{j=1}^n|x_j|^p\right)^{1/p}, \qquad x\in\doubleR^n,
\]
whereas $\norm{x}_\infty:=\max_{j=1,...,n}|x_j|$. We denote by $\langle\cdot,\cdot\rangle$ the canonical inner product of $\doubleR^n$, i.e. 
\[
\langle x,y\rangle = x^Ty=\sum_{j=1}^nx_jy_j, \qquad x,y\in\doubleR^n.
\]
If $x\in\doubleR^n$, $x^+$ is its positive part, i.e. $x^+\in\doubleR^n$ has coordinates $(x^+)_j=\max\{x_j,0\}$  ($j=1,...,n$). If $\Omega\subseteq\doubleR^n$, $\Omega^\perp$ denotes its orthogonal complement. For vectors $x,y\in\doubleR^n$, $x=(x_1,\ldots,x_n)$, $y=(y_1,\ldots,y_n)$, the notation $x\preceq y$ means that $x_j\leq y_j$ for every $j=1,\ldots,n$. Analogously, $x\prec y$ if $x_j<y_j$ for every $j=1,\ldots,n$. The relationships $x\succeq y$ and $x\succ y$ are defined similarly.}

 \textcolor{black}{We always consider $\doubleR^n$ endowed with the Euclidean topology. If $\Omega\subseteq\doubleR^n$, $\mathring{\Omega}$ denotes the interior of $\Omega$ and $\partial \Omega$ denotes the boundary of $\Omega$. If $g$ is a real-valued function defined on an open neighbourhood of $x_0\in\doubleR^n$, $\partial g(x_0)$ denotes the subdifferential of $g$ at $x_0$, see Section \ref{subsec:subdiff} below for the definition of subdifferential. Using the same notation to denote both the boundary of a set and the subdifferential of a function shall not cause confusion. If $\Omega\subseteq\doubleR^n$, $\agInt(\Omega)$ denotes the algebraic interior of $\Omega$, see Definition \ref{defAgint} below. If $g$ is a function and $\Omega$ is a subset of its domain, $g|_\Omega$ denotes the restriction of $g$ to $\Omega$. Finally, if $\Omega\subseteq\doubleR^n$, $\chi_\Omega$ denotes the characteristic function of $\Omega$.
}

\subsection{Lagrange duality} Consider a constrained optimization problem in the form:
\begin{equation}\label{eq:11}
	\text{minimize $F_0(x),$} \qquad \textcolor{black}{\Psi} x=y,\ F_l(x)\leq b_l, \ l=1,\ldots,M,
\end{equation}
where $\Psi\in\doubleR^{m\times n}$, $y\in\doubleR^{\textcolor{black}{m}}$ and $F_0,F_1,\ldots,F_M:\doubleR^n\to(-\infty,+\infty]$ are convex. We always assume that a minimizer of (\ref{eq:11}) exists. \\

A point $x\in\doubleR^n$ is called \textbf{feasible} if it belongs to the constraints, that is if:
\begin{equation}\label{defDiK}
	x\in K:=\Big\{\zeta\in\doubleR^n \ : \ \textcolor{black}{\Psi}\zeta=y \ and \ F_l(\zeta)\leq b_l, \ l=1,\ldots,M\Big\}
\end{equation}
and $K$ is called the \textbf{set of feasible points}. To avoid triviality, we always assume $K\neq\varnothing$, in which case problem (\ref{eq:11}) is called \textbf{feasible}. In view of the definition of $K$, problem (\ref{eq:11}) can be implicitly written as:
\[
	\text{minimize}\quad F_0(x), \qquad \text{$x\in K$}.
\]

\textcolor{black}{Convex problems} such as (\ref{lassettino}) and (\ref{constSVbasisP}) can be approached by considering their Lagrange formulation, see Subsection \ref{subsec:LF} below. The \textbf{Lagrange function} related to (\ref{eq:11}) is the function $L:\doubleR^n\times\doubleR^m\times [0,+\infty)^M\to(-\infty,+\infty]$ defined as:
\[
	L(x,\xi,\lambda):=F_0(x)+\langle\xi,\textcolor{black}{\Psi} x-y\rangle+\sum_{l=1}^M\lambda_l(F_l(x)-b_l).
\]

Observe that for all $\xi,\lambda$ and $x\in K$:
\[
	L(x,\xi,\lambda)=F_0(x)+\underbrace{\langle\xi,\textcolor{black}{\Psi} x-y\rangle}_\text{$=0$}+\sum_{l=1}^M\underbrace{\lambda_l}_\text{$\geq0$}(\underbrace{F_l(x)-b_l}_\text{$\leq0$})\leq F_0(x),
\]
so that:
\begin{equation}\label{eq:12}
	\inf_{x\in \doubleR^n}L(x,\xi,\lambda)\leq \inf_{x\in K}L(x,\xi,\lambda)\leq \inf_{x\in K}F_0(x).
\end{equation}

\begin{defi}
	The function $H:\doubleR^m\times[0,+\infty)^M\to[-\infty,+\infty]$ defined as:
	\[
		H(\xi,\lambda):=\inf_{x\in\doubleR^n}L(x,\xi,\lambda)
	\]
	is called \textbf{Lagrange dual function}.
\end{defi}

Inequalities (\ref{eq:12}) read as:
\begin{equation}\label{weak1}
	H(\xi,\lambda)\leq \inf_{x\in K}F_0(x)
\end{equation}
for all $\xi\in\doubleR^m$ and all $\lambda\in[0,+\infty)^M$. 
Stating (\ref{weak1}) differently, we have the \textbf{weak duality inequality}:
\begin{equation}\label{weak}\tag{W}
	\sup_{\underset{\lambda\succeq0}{\xi\in\doubleR^m}}H(\xi,\lambda)\leq \inf_{x\in K}F_0(x).
\end{equation}
We point out that (\ref{weak}) is equivalent to:
\begin{equation}\label{concaveW}
	\sup_{\xi,\lambda}\inf_xL(x,\xi,\lambda)\leq \inf_x\sup_{\xi,\lambda}L(x,\xi,\lambda)
\end{equation}
(see \cite[Subsection 5.4.1]{BV}). 

We are interested in computing the parameters $(\xi,\lambda)$ such that (\ref{weak}) is an equality, in  which case (\ref{weak}) becomes:
\begin{equation}\label{strong}\tag{S}
	\sup_{\underset{\lambda\succeq0}{\xi\in\doubleR^m}}H(\xi,\lambda)= \inf_{x\in K}F_0(x),
\end{equation}
\textcolor{black}{so that \textbf{strong duality} (\ref{strong}) holds for problem (\ref{eq:11})}.

\subsection{Subdifferential}\label{subsec:subdiff}
\begin{defi}[Subdifferential]
	Let $\Omega\subseteq\doubleR^n$ be open and $g:\Omega\to\doubleR$. Let $x_0\in\Omega$. The \textbf{subdifferential} of $g$ at $x_0$ is the set:
	\[
		\partial g(x_0):=\{v\in\doubleR^n \ : \ g(x)\geq g(x_0)+v^T(x-x_0) \ \forall x\in\Omega\}.
	\]
	We refer to any $v\in\partial g(x_0)$ as a \textbf{subgradient} of $g$ at $x_0$.
\end{defi}
We will use the following proposition.

\begin{prop}
	Let $\Omega\subseteq\doubleR^n$ be open and $g:\Omega\to\doubleR$ be convex and continuous on $\Omega$. Let $x_0\in\Omega$. Then, $\partial g(x_0)\neq\varnothing$.
\end{prop}

\subsection{Lagrange formulation of constrained problems}\label{subsec:LF}
\textcolor{black}{Under the notation above,} let $F(x):=(F_1(x),...,F_M(x))$. In the convex framework, if the constraint $F(x)\preceq b$ does not reduce to $F(x)=b$, namely if for all $l=1,\ldots,M$ the inequality $F_l(x)<b_l$ holds for some $x\in \doubleR^n$, then strong duality holds. 

\begin{theorem}[Cf. \cite{BV}, Section 5.3.2]\label{thm1}
Assume that $F_0,F_1,\ldots,F_M$ are convex functions defined on $\doubleR^n$. %and $F_0:\doubleR^n\to(-\infty,+\infty]$. 
Let $x^\#$ be such that $F_0(x^\#)=\inf_{x\in\doubleR^n}F_0(x)$. If:
\begin{enumerate}[(i)]
	\item there exists $\tilde x\in \doubleR^n$ such that $\Psi\tilde x=y$ and $F(\tilde x)\prec b$ or,
	\item in absence of inequality constraints, if $K\neq\varnothing$ (i.e. if there exists $\tilde x\in\doubleR^n$ such that $\Psi\tilde x=y$), 
\end{enumerate}
then, there exists $(\xi^\#,\lambda^\#)\in\doubleR^m\times[0,+\infty)^M$ such that $H(\xi^\#,\lambda^\#)=\sup_{\xi,\lambda}H(\xi,\lambda)$ and $H(\xi^\#,\lambda^\#)=F_0(x^\#)$.
\end{theorem}

The proof of Theorem \ref{thm1} contains the fundamental construction we will use in the next sections and we report it for this reason. We refer to \cite[Subsection 5.3.2]{BV} \textcolor{black}{for the complete proof}. First, we need a result from functional analysis, which is well-known as (geometrical) \textit{Hahn-Banach theorem}.

\begin{defi}[Separating hyperplane]
Consider two subsets $\scriptA,\scriptB\subseteq\mathds{R}^n$. A hyperplane $\Gamma:=\{x\in\doubleR^n \ : \ \langle \xi,x\rangle=\alpha\}$ satisfying:
\begin{equation}\label{sepHyp}
	\langle\xi,x\rangle\leq\alpha \textcolor{black}{<}\langle\xi,y\rangle, \qquad x\in \mathcal{A}, \quad y\in\scriptB,
\end{equation}
is a \textbf{separating hyperplane} between $\scriptA$ and $\scriptB$.
\end{defi}

\begin{theorem}[Cf. \cite{R3} Theorem 3.4]\label{HBtheorem}
Let $\mathcal{A},\mathcal{B}\subset\doubleR^n$ be two convex and disjoint subsets. If $\scriptB$ is open, there exists $\xi\in\doubleR^n$ and $\alpha\in\doubleR$ such that (\ref{sepHyp}) holds for all $x\in\mathcal{A}$ and all $y\in\mathcal{B}$. 
\end{theorem}

\begin{proof}[Idea of the proof of Theorem \ref{thm1}]
	First, one assumes that $A$ has full row-rank. Moreover, one reduces to consider the situation in which $p^\ast:=\inf_{x\in K}F_0(x)>-\infty$, otherwise the assertion is trivial.
	
	Consider the set:
	\begin{equation}\label{defdiG}
	\mathcal{G}:=\Big\{\left(F(x)-b,\Psi x-y,F_0(x)\right)\in\doubleR^M\times\doubleR^m\times\doubleR \ : \ x\in\doubleR^n\Big\},
\end{equation}
where, with an abuse of notation, $\Psi x-y$ denotes the row vector with the same (ordered) entries of $\Psi x-y$, and $\scriptA$ be defined as the epigraph:
	\begin{equation}\label{eqA1}\begin{split}
	\scriptA&:=\scriptG+((\doubleR_{\geq0})^M\times\doubleR^m\times\doubleR_{\geq0})=\\
	&=\Big\{(u,v,t)\in \doubleR^M\times\doubleR^m\times\doubleR \ : \ u\succeq F(x)-b,  \\
	& \ \ \ \ \ \  v= \Psi x-y, \ t\geq F_0(x) \ for \ some \ x\in\doubleR^n\Big\}.
\end{split}
\end{equation}

It is easy to verify that if $F_0,F_1,\ldots,F_M$ are convex, than $\scriptA$ is convex. Then, consider the set:
	\[
		\scriptB:=\Big\{(0,0,s)\in\doubleR^M\times\doubleR^m\times\doubleR \ : \ s<p^\ast\Big\}.
	\]
	$\scriptA$ and $\scriptB$ are clearly disjoint, $\scriptB$ (which is an open half-line) being trivially convex and open. Therefore, the assumptions of Theorem \ref{HBtheorem} are satisfied and we conclude that there exists a triple of parameters $(\tilde\lambda,\tilde\xi,\mu)\neq0$ and $\alpha\in\doubleR$ such that:
	\begin{align}
	\label{eq91}
	&(u,v,t)\in\scriptA \ \ \Longrightarrow \ \ \tilde\lambda^Tu+\langle\tilde\xi,v\rangle+\mu t\geq\alpha,\\
	\label{eq92}
	& (u,v,t)\in\scriptB \ \ \Longrightarrow \ \ \tilde\lambda^Tu+\langle\tilde\xi,v\rangle+\mu t\leq\alpha.
	\end{align}
	It is easy to see that the definition of $\scriptA$, together with (\ref{eq91}), imply that $\tilde\lambda_l\geq0$ for all $l=1,\ldots,M$ and $\mu\geq0$. Also, applying the definition of $\scriptB$ to (\ref{eq92}), one finds that $\mu t\leq\alpha$ for all $t<p^\ast$, which implies that $\mu p^\ast\leq\alpha$. Therefore, for all $x\in\doubleR^n$,
	\begin{equation}\label{eq111}
		\sum_{l=1}^M\tilde\lambda_l(F_l(x)-b_l)+\langle\tilde\xi,\Psi x-y\rangle+\mu F_0(x)\geq\alpha\geq \mu p^\ast.
	\end{equation}
	If $\mu>0$, then (\ref{eq111}) gives that $L(x,\tilde\xi/\mu,\tilde\lambda/\mu)\geq p^\ast$ for all $x\in\doubleR^n$, which implies that $H(\tilde\xi/\mu,\tilde\lambda/\mu)\geq p^\ast$. Since the other inequality  \textcolor{black}{holds trivially by the} weak duality inequality, we conclude that $H(\tilde\xi/\mu,\tilde\lambda/\mu)=p^\ast$.
	\textcolor{black}{Finally, using the assumptions on the rank of $\Psi$ and on the existence of a point satisfying the strict inequality constraint, one proves by contradiction that it must be $\mu>0$}.
\end{proof}

\textcolor{black}{\begin{defi}[Lagrange Multipliers]\label{defLOP}
	We refer to a couple $(\xi^\#,\lambda^\#)\in\doubleR^m\times[0,+\infty)^M$ as to \textbf{Lagrange multipliers} for the problem (\ref{eq:11}) if $(\xi^\#,\lambda^\#)$ attend the supremum in (\ref{strong}).
\end{defi}}

As a consequence of Theorem \ref{thm1}, we have the following result, which relates the minimizers of (\ref{eq:11}) and \textcolor{black}{those} of the dual problem $\max_{\xi,\lambda}H(\xi,\lambda)$, providing also the \textcolor{black}{Lagrange multipliers}, that may not be unique.

\begin{cor}[Cf. \cite{FR} Theorem B.28]\label{cor1}
	Let $F_0:\doubleR^n\to[0,+\infty)$ and $\phi:[0,+\infty)\to\doubleR$ be such that $\phi$ is monotonically increasing and $\phi\circ F_0$ is convex. Let $\tau_j>0$ ($j=1,\ldots,M$) and $\psi_j:\doubleR^n\to\doubleR$ ($j=1,\ldots,M$) be convex functions such that $\psi_j^{-1}([0,\tau_j))\neq\varnothing$ for all $j=1,\ldots,M$. Let $x^\#$ \textcolor{black}{be} a minimizer of the problem:
	\begin{equation}\label{prob12}
		\text{minimize}\quad F_0(x), \qquad x\in\doubleR^n \ \psi(x)\preceq\tau,
	\end{equation}
	\textcolor{black}{where} $\tau=(\tau_1,\ldots,\tau_M)$. Then, there exist $\lambda_j\geq0$ ($j=1,\ldots,M$) such that $x^\#$ is a minimizer of:
	\begin{equation}\label{prob13}
		\text{minimize}\quad \phi(F_0(x))+\sum_{j=1}^M\lambda_j\psi_j(x).
	\end{equation} 
\end{cor}
\begin{proof}
Since $\phi$ is monotonically increasing, (\ref{prob12}) is obviously equivalent to:
	\begin{equation}\label{prob14}
		\text{minimize}\quad\phi(F_0(x)), \qquad x\in\doubleR^n \ \psi_j(x)\leq\tau_j,
	\end{equation}
	($j=1,\ldots,M$) whose Lagrangian is given by:
	\begin{equation}\label{probmin}
		L(x,\lambda)=\phi(F_0(x))+\sum_{j=1}^M\lambda_j(\psi_j(x)-\tau_j).
	\end{equation}
	By the assumption, $\phi\circ F_0$ and each $\psi_j$ are convex and the inequalities $\psi_j(\tilde x)<\tau_j$ are satisfied by some $\tilde x\in\doubleR^n$ (observe that here we need $\tau_j>0$), so we can apply Theorem \ref{thm1} to get $H(\lambda^\#)=\phi(F_0(x^\#))$ for some $\lambda^\#\in[0,+\infty)^M$. By (\ref{concaveW}), for all $x\in\doubleR^n$:
	\[
		L(x^\#,\lambda^\#)\leq L(x,\lambda^\#),
	\]
	so that $x^\#$ is also a minimizer of the function $x\in\doubleR^n\mapsto L(x,\lambda^\#)$. Since the constant terms $-\lambda_j\tau_j$ in (\ref{probmin}) do not affect the \textcolor{black}{set of} minimizers, we have that $x^\#$ is a minimizer of:
	\[
		\text{minimize} \quad \phi(F_0(x))+\sum_{j=1}^M\lambda_j^\#(\psi_j(x)-\tau), \qquad x\in\doubleR^n.
	\]
\end{proof}

\begin{oss}
	Theorem \ref{HBtheorem} has a complex version that holds with $\Re\langle z,w\rangle=\Re\left(\sum_{j=1}^n\overline{z_j}w_j\right)$ ($\Re$ denotes the real part of a complex number) instead of $\langle \cdot,\cdot\rangle$. In particular, the entire theory presented in this work is applicable in the complex framework as well. This extension involves replacing the canonical real inner product of $\doubleR^n$ with the real inner product on $\mathds{C}^n$ defined above. Therefore, we do not need to study the complex case separately, as only the structure of $\mathds{C}^n$ as a real vector space is involved.
\end{oss}

\begin{oss}
	To sum up, Theorem \ref{thm1} and Corollary \ref{cor1} together tell that, up to the sign, the coefficients of any hyperplane \textcolor{black}{separating} the two sets:
	\[
		\scriptA=\Big\{(u,t)\in\doubleR^{M+1} \ : \ u\succeq F(x)-b, \ t\geq F_0(x) \ for \ x\in\doubleR^n\Big\}
	\]
	and 
	\[
		\scriptB=\Big\{(0,t)\in\doubleR^{M+1} \ : \ t<\inf_{x\in K}F_0(x)\Big\}
	\]
	define Lagrange multipliers for problem (\ref{eq:11}), in absence of equality constraints, i.e. if $y=0$ and $\Psi=0$ in (\ref{eq:11}). This is the geometric idea \textcolor{black}{that} we will apply in the following sections to the \textcolor{black}{weighted} LASSO.
\end{oss}

\section{The \textcolor{black}{weighted} LASSO}\label{sec:LASSO}
Let $A\in\doubleR^{m\times n}$, $b\in\doubleR^m$ and $\tau_1,\ldots,\tau_n\geq0$. We denote with $a_{\ast,j}$ the $j$-th column of $A$ and set $b=(b_1,\ldots,b_m)$. We consider the constrained minimization problem:
\begin{equation}\label{eq:01}
	\text{minimize} \quad \norm{Ax-b}_2^2, \qquad x\in\doubleR^n, \ |x_j|\leq\tau_j, \  {j=1,\ldots,n}.
\end{equation}
We also assume that $\tau_j\neq0$ for all $j=1,\ldots,n$. In fact, if $\tau_j=0$ for some $j=1,\ldots,n$, then the solution $x=(x_1,\ldots,x_n)$ has $x_j=0$. In this case, problem (\ref{eq:01}) reduces to
\begin{equation}\label{equivalenceOfTau0}
		\text{minimize} \quad \norm{\tilde Ay- b}_2^2,\qquad  y\in\doubleR^{n-r}, \ |y_{i_j}|\leq\tau_{i_j}, \ {j=1,\ldots,n-r},
\end{equation}
where $r=\text{card}\{j : \tau_j=0\}\leq m$, $J=\{1\leq i_1<\ldots<i_{n-r}\leq n\}:=\{j : \tau_j\neq 0\}$ and $\tilde A=(a_{\ast,j})_{j\in J}\in\mathds{R}^{m\times (n-r)}$.\\

Let $K$ denote the set of the feasible points of problem (\ref{eq:01}), that is:
\begin{equation}\label{K}
	K=\{x\in\doubleR^n \ : \ |x_j|\leq\tau_j \ \forall j=1,\ldots,n\}
\end{equation}
and consider the Lagrange function \textcolor{black}{associated to (\ref{eq:01})}, i.e.
\begin{equation}
	\label{eq22}
	L(x,\lambda_1,\ldots,\lambda_n)=\norm{Ax-b}_2^2+\sum_{j=1}^n\lambda_j(|x_j|-\tau_j).
\end{equation}
We are interested in a vector of Lagrange multipliers $\lambda^\#\succeq0$ for (\ref{eq:01}). Based on the proofs of Theorem \ref{thm1} and Corollary \ref{cor1}, $\lambda^\#$ can be chosen as the direction of \textcolor{black}{any} hyperplane \textcolor{black}{separating the sets}:
\begin{equation}\label{defDiA1}
\begin{split}
\scriptA &= \Big\{(u,t)\in \doubleR^n\times\doubleR \ : \ u_l\geq |x_l|-\tau_l \ (l=1,\ldots,n), \\
        & \qquad \quad t\geq \norm{Ax-b}_2^2 \text{ for some } x\in\doubleR^n\Big\}
\end{split}
\end{equation}

and
\begin{equation}\label{defDiB1}
\scriptB = \Big\{(0,t)\in\mathds{R}^{n}\times\mathds{R} \ : \ t<p^\ast\Big\}
\end{equation}

where $p^\ast:=\inf_{x\in K}\norm{Ax-b}_2^2$. 

\subsection{The scalar case}
\label{subsec:31}
To clarify the general procedure, we focus on the simple case $m=n=1$ first, in which  (\ref{eq:01}) becomes:
\begin{equation}\label{scalare}
	\text{minimize} \quad (Ax-b)^2, \qquad x\in\doubleR, \ |x|\leq\tau,
\end{equation}
where $A\in\doubleR\setminus\{0\}$ and $b\in\doubleR$. To find the Lagrange multipliers, we consider the set $\scriptG$ of points $(u,t)\in\doubleR^2$ that satisfy:
\[
	\begin{cases}
		u= |x|-\tau,\\
		t= (Ax-b)^2,
	\end{cases}
\]
which give a curve of the half-plane $U=\{(u,t)\in\doubleR^2 \ : \ u\geq\textcolor{black}{-}\tau, \ t\geq0\}$ parametrized by $x\in\doubleR$. More precisely:% (see Figure \ref{fig:1}):
\begin{itemize}
	\item if $x\geq0$, 
	\[
		\begin{cases}
		x=u+\tau,\\
		t=\big(A(u+\tau)-b\big)^2=(Au+(A\tau-b))^2,
	\end{cases}
	\]
	which is a branch of parabola in $U$ with vertex in $(\frac{b}{A}-\tau,0)$.
	\item If $x<0$
		\[
		\begin{cases}
		x=-u-\tau,\\
		t=\big(-A(u+\tau)-b\big)^2=(Au+(A\tau+b))^2,
	\end{cases}
	\]
	which is, again, a branch of parabola in $U$, having its vertex in $(-\frac{b}{A}-\tau,0)$.
\end{itemize}

\begin{prop}\label{propCaso1}
	Let $\tau>0$, $A\in\doubleR\setminus\{0\}$, $b\in\doubleR$. A Lagrange multiplier for (\ref{scalare}) is given by:
	\[
		\lambda^\#=\begin{cases}
			2A^2(|b/A|-\tau) & \text{if $0<\tau<|b/A|$},\\
			0 & \text{if $\tau\geq|b/A|$}
		\end{cases}=2A^2(|b/A|-\tau)^+.
	\]
	\textcolor{black}{Namely, if $x^\#$ is a minimizer of (\ref{scalare}), then it is also a minimizer for the problem:
	\[
		\text{minimize} \quad (Ax-b)^2+\lambda^\#|x|, \qquad x\in\doubleR.
	\]}
\end{prop}

\subsection{Properties of $\scriptA$}\label{subsec:33}

Consider $A\in\doubleR^{m\times n}$ and $b=(b_1,...,b_m)\in\doubleR^m$, with:
\[
	A=\begin{pmatrix}
		a_{11} & \ldots & a_{1n}\\
		\vdots & \ddots & \vdots\\
		a_{m1} & \ldots & a_{mn}
	\end{pmatrix}.
\]

We consider the problem (\ref{eq:01}) and the associated Lagrange function:
\begin{equation}\label{lagrangiangeneral}
	L(x,\lambda):=\norm{Ax-b}_2^2+\sum_{j=1}^n\lambda_j(|x_j|-\tau_j).
\end{equation}
Recall that $p^\ast$ was defined as $p^\ast:=\min_{x\in K}\norm{Ax-b}_2^2$, being $K$ the set of the points $x\in \textcolor{black}{\mathds{R}^n}$ such that $|x_j|\leq\tau_j$ for all $j=1,\ldots,n$. It is not difficult to verify that:
\begin{equation}\label{pstar}
	p^\ast=\inf\Big\{ t\in\doubleR \ : \ (u,t)\in\scriptG, \ u_j\leq0 \ \ \forall j=1,\ldots,n\Big\}.
\end{equation}

Let $\mathcal{M}_n$ be the set of the $n$-dimensional signature matrices, that are the diagonal matrices $S=(s_{ij})_{i,j=1}^n\in\doubleR^{n\times n}$ such that \textcolor{black}{$|s_{jj}|=1$} for all $j=1,\ldots,n$.
Observe that if $S\in\mathcal{M}_n$, then $S^2=I_{n\times n}$, where $I_{n\times n}$ denotes the identity matrix in $\doubleR^{n\times n}$, in particular $S$ is invertible with $S^{-1}=S$. If $x\in\doubleR^n$ and $S\in\mathcal{M}_n$ is such that $Sx\in\prod_{j=1}^n[0,+\infty)$, we write $S\in \sgn(x)$.

\begin{lemma}\label{studioKernel} Let $A\in\mathds{R}^{m\times n}$, $b\in\mathds{R}^m$ and $\tau_j>0$ for $j=1,\ldots,n$.
	Let $S\in\mathcal{M}_n$. There exists $u\in\prod_{j=1}^n[-\tau_j,0]$ such that $ASu+AS\tau-b=0$ if and only if $S\in \sgn(x)$ for some $x\in\doubleR^n$ such that $Ax=b$ and $|x_j|\leq\tau_j$. 
	%\end{enumerate}
\end{lemma}
\begin{proof}
	Assume that there exists $u\in\prod_{j=1}^n[-\tau_j,0]$ such that $ASu+AS\tau-b=0$ and let $x:=S(u+\tau)$. Then, $Sx=u+\tau\in\prod_{j=1}^n[0,\tau_j]$, so that $S\in \sgn(x)$, $|x_j|\leq\tau_j$ for all $j=1,\ldots,n$ and
	\[
		0=AS(u+\tau)-b=Ax-b.
	\]
	Vice versa, assume that $Ax=b$ for some $x\in\prod_{j=1}^n[0,\tau_j]$. Let $S\in \sgn(x)$ and $u:=Sx-\tau$. Then, $u\in\prod_{j=1}^n[-\tau_j,0]$ and
	\[
		0=Ax-b=A(Su+\tau)-b=ASu+AS\tau-b.
	\]
\end{proof}

Recall the definitions of the two sets $\scriptA$ and $\scriptB$ \textcolor{black}{given in (\ref{defDiA1}) and (\ref{defDiB1}) respectively}. First, \textcolor{black}{if} $\scriptG$ is the set of the points $(u,t)\in\doubleR^{n+1}$ such that:
\begin{equation}\label{eqdefG}
	\begin{cases}
		u_j= |x_j|-\tau_j & \text{$j=1,\ldots,n$},\\
		t=\norm{Ax-b}_2^2,
	\end{cases}
\end{equation}
for some $x\in\doubleR^n$, \textcolor{black}{then} 
\[
	\scriptA=\scriptG+[0,+\infty)^{n+1},
\]
that is, $(u,t)\in\scriptA$ if and only if 
\begin{equation}\label{eqdefA}
	\begin{cases}
		u_j\geq |x_j|-\tau_j & \text{$j=1,\ldots,n$},\\
		t\geq\norm{Ax-b}_2^2,
	\end{cases}
\end{equation}
for some $x\in\doubleR^n$. Finally, $(u,t)\in\scriptB$ if and only if $t<p^\ast=\min_{x_j\leq\tau_j}\norm{Ax-b}_2^2$. \\

We will prove that the equations (\ref{eqdefG}) defining $\mathcal{G}$ can be written in terms of $\mathcal{M}_n$. 

\begin{lemma}\label{esplG}
Let $\tau_1,\ldots,\tau_n>0$ and \textcolor{black}{let $\scriptG$ be the set of points satisfying (\ref{eqdefG})}.
Then, 
\begin{enumerate}[(i)]
\item $\scriptG$ is closed.
\item $(u,p^\ast)\in\scriptG$ for some $u\in\doubleR^n$ such that $-\tau_j\leq u_j\leq 0$ for all $j=1,\ldots,n$. Moreover, $p^\ast=\min\Big\{ t\in\doubleR \ : \ (u,t)\in\scriptG, \ u_j\leq0 \ \ \forall j=1,\ldots,n\Big\}$.
\item For every $(u,t)\in\scriptG$ there exists $S\in\scriptM_n$ such that $t=\norm{ASu+(AS\tau-b)}_2^2$. Viceversa, if $t=\norm{ASu+(AS\tau-b)}_2^2$ for some $u\in\doubleR^n$ such that $u_j\geq -\tau_j$ and some $S\in\mathcal{M}_n$, then $(u,t)\in\scriptG$.
\end{enumerate}
\end{lemma}
\begin{proof}
	We prove that $\scriptG$ is closed. For, let $(u^k,t^k)\in\scriptG$ converge to $(u,t)\in\doubleR^{n+1}$. We prove that $(u,t)\in\scriptG$. Let $x^k\in\doubleR^n$ be such that (\ref{eqdefG}) is satisfied for $(u^k,t^k)$. Then, $|x^k_j|=u_j^k-\tau_j\leq u_j+1-\tau_j$ for $j$ sufficiently large. In particular, the sequence $\{x^k\}_k$ is bounded and, thus, it converges up to subsequences. Without loss of generality, we may assume that $(x^k)_k$ converges to $x:=\lim_{k\to+\infty}x^k$ in $\doubleR^n$. Then, for all $j=1,\ldots,n$,
	\[
		|x_j|=\lim_{k\to+\infty}|x_j^k|=\lim_{k\to+\infty}u_j^k+\tau_j=u_j+\tau_j
	\]
	and, by continuity,
	\[
		\norm{Ax-b}_2^2=\lim_{k\to+\infty}\norm{Ax^k-b}_2^2=\lim_{k\to+\infty}t^k=t.
	\]
	This proves that $(u,t)\in\scriptG$ and, thus, that $\scriptG$ is closed. (ii) follows by (i) and (\ref{pstar}).
	
	It remains to check (iii). If $(u,t)\in\scriptG$, there exists $x\in\doubleR^n$ satisfying (\ref{eqdefG}). Let $S\in\mathcal{M}_n$ be such that $|x|=Sx$, where $|x|:=(|x_1|,\ldots,|x_n|)$. Then, using the fact that $S^{-1}=S$,
	\[
		|x|=u+\tau \ \ \ \Longrightarrow \ \ \ Sx=(u+\tau) \ \ \ \Longrightarrow \ \ \ x=S(u+\tau).
	\]
	By the last equation of (\ref{eqdefG}), we have:
	\[
		t=\norm{Ax-b}_2^2=\norm{ASu+(AS\tau-b)}_2^2.
	\]
	Viceversa, assume that $t=\norm{ASu+(AS\tau-b)}_2^2$ for some $S\in\mathcal{M}_n$ and $u\in\doubleR^n$ is such that $u\succeq-\tau$. Let $x:=S(u+\tau)$, then $|x_j|=|u_j+\tau_j|=u_j+\tau_j$ for all $j=1,\ldots,n$ and $t=\norm{Ax-b}_2^2$. This proves that $(u,t)\in\scriptG$ and the proof of (iii) is concluded.
\end{proof}

\begin{lemma}\label{Ggiù}
	Let $u\in\prod_{j=1}^n[-\tau_j,+\infty)$,
	\begin{equation}\label{defdihG}
		h_G(u):=\min_{S\in\mathcal{M}_n}\norm{ASu+AS\tau-b}_2^2
	\end{equation}
	and 
	\[
		g_G(u):=\min_{(u,s)\in\scriptG}s.
	\]
	Then, $h_G(u)=g_G(u)$.
\end{lemma}
\begin{proof}
	By Lemma \ref{esplG} (iii), if $(u,s)\in\scriptG$, then $s=\norm{AS_0u+AS_0\tau-b}_2^2$ for some $S_0\in\mathcal{M}_n$. Hence,
	\[
		h_G(u)=\min_{S\in\scriptM_n}\norm{ASu+AS\tau-b}_2^2\leq\norm{AS_0u+AS_0\tau-b}_2^2=s
	\]
	for all $s$ such that $(u,s)\in\scriptG$. \textcolor{black}{Taking} the minimum, we get $h_G(u)\leq g_G(u)$. On the other hand, $(u,h_G(u))\in\scriptG$ by Lemma \ref{esplG} (iii). Therefore, $g_G(u)\leq h_G(u)$ by definition of $g_G$.
\end{proof}

\begin{lemma}\label{lemmaAchiuso} \textcolor{black}{Let $\scriptG$ be the set of points satisfying (\ref{eqdefG}) and $\scriptA$ be the set of points satisfying (\ref{eqdefA}). Then,}
	%Let $\scriptA:=\scriptG+[0,+\infty)^{n+1}$, where $\scriptG$ is defined as in Lemma \ref{esplG}. Then, \begin{enumerate}[(i)]
	\begin{enumerate}[(i)]
	\item $\scriptG\subseteq\scriptA$;
	\item $\scriptA$ is closed.
\end{enumerate}
\end{lemma}
\begin{proof}
(i) is obvious. We prove (ii).

Let $(u^k,t^k)\in\scriptA$ be a sequence such that $(u^k,t^k)\xrightarrow[k\to+\infty]{}(u,t)$ in $\doubleR^{n+1}$. We need to prove that $(u,t)\in\scriptA$. For all $k$, let $x^k\in\doubleR^n$ be such that:
\[
	\begin{cases}
			u^k_1\geq|x^k_1|-\tau_1,\\
			\vdots\\
			u_n^k\geq |x^k_n|-\tau_n,\\
			t^k\geq\norm{Ax^k-b}_2^2.
		\end{cases}
\]

The sequence $\{x^k\}_k$ is bounded, in fact for all $j=1,\ldots,n$, $|x_j^k|\leq u_j^k+\tau_j\leq u_j+1+\tau_j$ for $k$ sufficiently large. Therefore, up to subsequences, we can assume $x^k\xrightarrow[k\to+\infty]{}x$ in $\doubleR^n$. For all $j=1,\ldots,n$,
\[
|x_j|=\lim_{k\to+\infty}|x_j^k|\leq\lim_{k\to+\infty} u_j^k+\tau_j=u_j+\tau_j.
\]
Moreover, by continuity,
\begin{align*}
	\norm{Ax-b}_2^2&=\lim_{k\to+\infty}\norm{Ax^k-b}_2^2\leq \lim_{k\to+\infty}t^k=t.
\end{align*}
\end{proof}

\begin{lemma}\label{lemmapropg}\textcolor{black}{Let $\scriptA$ be the set of points satisfying (\ref{eqdefA}).}
\begin{enumerate}[(i)]
\item $\scriptA$ is the epigraph of a convex non-negative function function $g:\prod_{j=1}^n[-\tau_j,+\infty)\to\doubleR$ which is continuous in $\prod_{j=1}^n(-\tau_j,+\infty)$;
\item $\partial g(0)\neq\varnothing$; 
\item $g(u)=0$ if and only if $(u,t)\in\scriptA$ for all $t\geq0$.
\end{enumerate}
\end{lemma}
\begin{proof}
First, observe that $\scriptA\subseteq\{(u,t) \ : \ t\geq0\}$ since $t\geq \norm{Ax-b}_2^2\geq0$ for some $x\in\doubleR^n$ whenever $(u,t)\in\scriptA$.

For the sake of completeness, we check that $\scriptA$ is the epigraph of the function:
	\begin{equation}\label{defBordoA}
	g(u)=\min_{(u,s)\in\scriptA}s \ \ \ \ \Big(u\in\prod_j[-\tau_j,+\infty)\Big),
	\end{equation}
	which is well defined by Lemma \ref{lemmaAchiuso}. 
	
	By the observation at the beginning of the proof, $g(u)\geq0$. 
	Let
	\[
		epi(g):=\{(u,t) \ : \ t\geq g(u)\}
	\]
	be the epigraph of $g$. If $(u,t)\in\scriptA$, then $t\geq \min_{(u,s)\in\scriptA}s=g(u)$, this means that $(u,t)\in epi(g)$. On the other hand, if $(u,t)\in epi(g)$, then $t\geq s$ for some $(u,s)\in\scriptA$. But, if $t\geq s$ (and $(u,s)\in\scriptA$), then $(u,t)\in \scriptA$ as well, since $\scriptA$ contains the vertical upper half-lines having their origins in $(u,s)$, namely $(u,s)+(\{0\}\times[0,+\infty))$.
	
	This proves that $\scriptA$ is an epigraph. Moreover, $g$ is convex because $\scriptA$ is convex  (see \cite{RW} Proposition 2.4). The continuity of $g$ on $\prod_j(-\tau_j,+\infty)$ follows from \cite{R}, Theorem 10.1.  
	This proves (i). 
	
	Moreover, since $\tau_j>0$ for all $j=1,\ldots,n$, $0\in\doubleR^n$ is an interior point of $\prod_j[-\tau_j,+\infty)$. Since $g$ is continuous and convex in $\prod_j(-\tau_j,+\infty)$, the subdifferential of $g$ in $0$ is non-empty and (ii) follows.
	
	To prove (iii), assume that $g(u)=0$. Then, $\min_{(u,s)\in\scriptA}s=0$ implies $(u,0)\in\scriptA$. Since for all $t\geq0$, $(u,0)+(\{0\}\times [0,+\infty))\in\scriptA$, we have that $(u,t)\in\scriptA$ for all $t\geq0$. For the converse, assume that $(u,t)\in\scriptA$ for all $t\geq0$. Then, $(u,0)\in\scriptA$, so that (by the non-negativity of $g$) $0\leq g(u)\leq0$. This proves the equivalence in (iii).
	\end{proof}
	
	\begin{oss} As we observed in the general theory situation, $(0,s)\in\scriptA$ if and only if $s\geq p^\ast$. This tells that $g(0)=p^\ast$ and $(0,p^\ast)\in\scriptA$. 
	\end{oss}
	
	We want to prove formally that $g(u)$ defines the boundary $\partial\scriptA$ of $\scriptA$ in a neighborhood of $u=0$ and, then, find an explicit formula for $g(u)$. Observe that, $\scriptA=\partial\scriptA\cup \mathring{\scriptA}$, where $\mathring{\scriptA}$ denotes the topologic interior of $\scriptA$. Since $\scriptA$ is closed and convex in $\doubleR^n$, $\mathring{\scriptA}$ coincides with the algebraic interior of $\scriptA$, which is defined as follows:
	
	\begin{defi}\label{defAgint}Let $X$ be a vector space and $\scriptA\subseteq X$ be a subset. The \textbf{algebraic interior} of $\scriptA$ is defined as:
	\[
		\agInt(\scriptA):=\{a\in \scriptA \ : \ \forall x\in X \ \exists \epsi_x>0 \ s.t. \ a+tx\in\scriptA \ \forall t\in(-\epsi_x,\epsi_x) \}.
	\]
	\end{defi}
	
	\begin{lemma}
		Let $\scriptA$ be as in Lemma \ref{lemmaAchiuso}. Then,
		\begin{equation}\label{bordoA}\begin{split}
			\partial \scriptA=&\{(u,t)\in\scriptA \ : \ t=g(u), \ u_j>-\tau_j \ \forall j=1,\ldots, n\}\cup\\
			&\cup\{(u,t)\in\scriptA \ : \ u_j=-\tau_j \ for \ some \ j=1,\ldots,n\}
		\end{split}
		\end{equation}
		and the union is disjoint. Moreover, 
		\[
			\{(u,t)\in\scriptA \ : \ u_j=-\tau_j \ for \ some \ j=1,\ldots,n\}=\{(u,t)\in\partial\scriptA \ : \ (u,t+\alpha)\in\partial\scriptA \ \forall \alpha\geq0\}.
		\]
	\end{lemma}
	\begin{proof} Observe that the union in (\ref{bordoA}) is clearly disjoint. We first prove (\ref{bordoA}).
		\begin{itemize}
		\item[($\supseteq$)] \textcolor{black}{None} of the sets on the RHS of (\ref{bordoA}) is contained in $\mathring{A}$. In fact,
		\begin{itemize}
		\item[$\bullet$] the definition of $g(u)$ implies that for all $\epsi>0$, $-\epsi<t<\epsi$, $(u,t)\in\scriptA$ if and only if $t\geq0$, so that $(u,g(u))\notin\agInt(\scriptA)=\mathring{\scriptA}$. This proves that the graph of $g$ in $\prod_j(-\tau_j,+\infty)$ is a subset of $\partial \scriptA$.
		\item[$\bullet$] Analogously, assume that $u_j=-\tau_j$ for some $j=1,\ldots,n$, and for all $\epsi>0$ consider the point $(u_\epsi,t)$, where $(u_\epsi)_l=u_l$ for all $l\neq j$ and $(u_\epsi)_j=-\tau_j-\epsi$. But $g$ is defined on $\prod_j[-\tau_j,+\infty)$ and $\scriptA$ is its epigraph, hence all the points of $\scriptA$ must be in the form $(u,g(u)+\alpha)$ for some $u\in\prod_j[-\tau_j,+\infty)$, $t=g(u)+\alpha$ ($\alpha\geq0$), hence $(u_\epsi,t)\notin\scriptA$ and this proves that $(u,t)\notin \agInt(\scriptA)$.
		\end{itemize}
		The fact that $\partial\scriptA=\scriptA\setminus\mathring{\scriptA}$ proves the first inclusion.
		\item[($\subseteq$)] We prove that the complementary of the RHS of (\ref{bordoA}) in $\doubleR^{n+1}$ is contained in $\mathring{A}$. Let $(u,t)$ be such that $u>-\tau_j$ for all $j$ and $t>g(u)$ (as it is easy to check, these are the conditions for $(u,t)$ to belong to the complementary of the union of the two set at the LHS of (\ref{bordoA})). 

		Let $d:=t-g(u)>0$. Since $g$ is continuous on $\prod_j(-\tau_j,+\infty)$, there exists $\delta>0$ such that $|g(u)-g(v)|<d/4$ for all $v\in B_\delta(u):=\{w\in\doubleR^n \ : \ |w-u|<\delta\}$. In particular, for all $v\in B_\delta(u)$, $g(v)<t-\frac{3}{4}d<t$. Then, $B_\delta(u)\times(t-\frac{3}{4}d,+\infty)$ is all contained in $\scriptA$ (because $\scriptA$ is the epigraph of $g$) and it is an open neighborhood of $(u,t)$. Hence, $(u,t)\in\mathring{A}=\mathcal{A}\setminus\partial\scriptA$. 

		\end{itemize}
		Next, we check the second part of the lemma:
		\begin{itemize}
			\item[($\subseteq$)] assume $(u,t)\in\scriptA$ is such that $u_j=-\tau_j$ for some $j$. Then, by the first part of this Lemma, $(u,t+\alpha)\in\partial\scriptA$ for all $\alpha\geq0$, since (\ref{bordoA}) is a partition of $\partial\scriptA$.
			\item[($\supseteq$)] Assume that $(u,t+\alpha)\in\partial\scriptA$ for all $\alpha\geq0$. Then, $(u,t)\in\partial\scriptA$. Assume by contradiction that $u_j>-\tau_j$ for all $j$. Then, since (\ref{bordoA}) is a partition of $\partial\scriptA$, $g(u)=t+\alpha$ for all $\alpha\geq0$, which cannot be the case. 
		\end{itemize}
	\end{proof}
	
\textcolor{black}{The function $g$, defined in Lemma \ref{lemmapropg}, is expressed in terms of the function $h_G$, as shown in the following result.}
	
	\begin{theorem}\label{hg}
		\textcolor{black}{Let $\scriptA$ be the set of points satisfying (\ref{eqdefA}), $h_G$ and $g$ be the functions defined in (\ref{defdihG}) and (\ref{defBordoA}), respectively}. For $u\in\prod_{j=1}^n[-\tau_j,+\infty)$, $u=(u_1,\ldots,u_n)$, let $Q(u):=\prod_{j=1}^n[-\tau_j,u_j]$ and
\begin{equation}\label{esplg}
			h(u):=\min_{S\in\mathcal{M}_n, \ v\in Q(u)}\norm{AS(v+\tau)-b}_2^2=\min_{v\in Q(u)}h_G(v).
		\end{equation}
		Then, $h(u)=g(u)$ for all $u\in\prod_j[-\tau_j,+\infty)$.
	\end{theorem}
	\begin{proof}
	We first prove that $g(u)\leq h(u)$. For, it is enough to prove that $(u,h(u))\in\scriptA$, so that $g(u)\leq h(u)$ would follow by the definition of $g$. By definition of $h$, there exist $S_0\in\scriptM_n$ and $v\in Q(u)$ \textcolor{black}{so that}:
	\[
		h(u)=\norm{AS_0v+AS_0\tau-b}_2^2. 
	\]
	By Lemma \ref{esplG} (iii), $(v,h(u))\in\scriptG$. Since $u_j\geq v_j$ for all $j=1,\ldots,n$, it follows that $(u,h(u))\in\scriptA$ by definition of $\scriptA$.
	
	For the converse, since $(u,g(u))\in\scriptA$, there exists $(v',t)\in\scriptG$ such that $v_j'\leq u_j$ for all $j=1,\ldots,n$ and $g(u)\geq t$. In particular, $v'\in Q(u)$. By Lemma \ref{esplG} (iii), $t=\norm{AS_1v'+AS\tau-b}_2^2$ for some $S_1\in\scriptM_n$. Therefore,
	\[
		g(u)\geq \norm{AS_1v'+AS_1\tau-b}_2^2\geq \min_{S\in\mathcal{M}_n, \ v\in Q(u)}\norm{ASv+AS\tau-b}_2^2= h(u).
	\]
	This concludes the proof.

\end{proof}

Even if $g=h$, in what follows we still distinguish $h$ and $g$ when we want to stress the explicit definitions of both. Namely, we write $g(u)$ when we refer to $\min_{(u,s)\in \scriptA}s$ and $h(u)$ when we refer to (\ref{esplg}).

\begin{cor}
	Under the same notation as above, %for all $S\in\mathcal{M}_n$,
	\begin{equation}\label{formulaSenzaS}
		g(u)=\min_{-u-\tau\preceq v\preceq u+\tau}\norm{Av-b}_2^2.
	\end{equation}
\end{cor}
\begin{proof}
	Using the second expression in (\ref{esplg}),
	\[
		g(u)=\min_{S\in\mathcal{M}_n}\min_{v\in Q(u)}\norm{AS(v+\tau)-b}_2^2.
	\]
	But,
	\[
		f_S(v)=\norm{AS(v+\tau)-b}_2^2=f(S(v+\tau)),
	\]
	for $f(v)=\norm{Av-b}_2^2$, that gives:
	\[
		\min_{-\tau\preceq v\preceq u}f_S(v)=\min_{v\in Q(u)}f(S(v+\tau))=\min_{v\in S(Q(u)+\tau)}\norm{Av-b}_2^2,
	\]
	so that:
	\[
		\min_{S\in\mathcal{M}_n}\min_{-\tau\preceq v\preceq u}f_S(v)=\min_{\bigcup_{S\in\mathcal{M}_n}S(Q(u)+\tau)}\norm{Av-b}_2^2
	\]
	and the assertion follows by observing that
	\[
		\bigcup_{S\in\mathcal{M}_n}S(Q(u)+\tau)=\{v\in\mathds{R}^n: -u-\tau\preceq v\preceq u+\tau\}.
	\]
\end{proof}

\subsection{A result under conditions on the gradient of $\norm{Ax-b}_2^2$}

In general, the geometry of $\scriptA$ is so complicated that expressing $g$ explicitly may turn into a tough task. Nevertheless, it is obvious that if $u$ is itself one of the minimizers of (\ref{esplg}), then $g(u)=h_G(u)=\min_{S\in\mathcal{M}_n}\norm{ASu+AS\tau-b}_2^2$. So, under further assumptions on $\nabla(\norm{ASu-b}_2^2)$ granting the equality $g(u)=h_G(u)$ holds in a neighborhood of $0$, we can compute explicitly the Lagrange multipliers. 

\begin{theorem}\label{thm2}
	Let $f(v)=\norm{Av-b}_2^2$ and assume that for all $k=1,\ldots,n$ the condition:
	\begin{equation}\label{condizionesullederivate}
		\sum_{j=1}^nu_j\langle a_{\ast,j},a_{\ast,k}\rangle\leq \langle b,a_{\ast,k}\rangle \qquad (-\tau\preceq u\preceq \tau)
	\end{equation}
	holds. Then, $g(u)=f(u+\tau)$ for all $u\in Q(0)$ and $\lambda^\#=A^T(b-A\tau)$ is a set of Lagrange multipliers for problem (\ref{eq:01}).
\end{theorem}
\begin{proof}
	The set of conditions (\ref{condizionesullederivate}) is equivalent to \textcolor{black}{$(Au-b)^TA\preceq 0$} for all $-\tau\preceq u\preceq \tau$, that is $\nabla f(u)\preceq 0$ for $-\tau\preceq u\preceq\tau$. We prove that, under this further condition, $g(u)=f(u+\tau)$ for all $u\in Q(0)$. Let $u\in Q(0)$ and $\mathfrak{n}\succ0$ be a unit vector. For all $t\in\mathds{R}$, define:
	\[\begin{split}
		f_{\mathfrak{n}}(t)&:=f(u+\tau+t\mathfrak{n})=\norm{A(u+\tau+t\mathfrak{n})-b}_2^2\\
		&=\norm{A\mathfrak{n}}_2^2t^2+2\langle A(u+\tau)-b , A\mathfrak{n} \rangle t +\norm{A(u+\tau)-b}_2^2,
	\end{split}\]
	which is the restriction of $f$ to the line $\{u+t\mathfrak{n} : t\in\mathds{R}\}$. If $\mathfrak{n}\in\text{ker}(A)$, then $f_{\mathfrak{n}}\equiv0$ and it has a global minimum in $t=0$. Assume $\mathfrak{n}\notin\text{ker}(A)$. The intersection of this line with $\{-\tau\preceq v\preceq \tau\}$ is contained in $(-\infty,0]$. If we prove that, for all $\mathfrak{n}\succ0$, $f_{\mathfrak{n}}$ has a constrained minimum in $t=0$, we get the first assertion. For, it's enough to observe that 
	\[
		f'_{\mathfrak{n}}(0)=\nabla f(u+\tau)\cdot\mathfrak{n}\leq0,
	\]
	because if $u\in Q(0)$, then $\{-u-\tau\preceq v \preceq u+\tau\}\subseteq \{-\tau\preceq v\preceq \tau\}$. This proves that $g(u)=f(u+\tau)$ for all $u\in Q(0)$. In particular,
	\[
		-\nabla g(0)=-\nabla f(\tau)=(b-A\tau)^TA\succeq0
	\]
	is a set of Lagrange multipliers for (\ref{eq:01}).
\end{proof}
\begin{oss}
It is not difficult to generalize Theorem \ref{thm2} a bit further. If the hyperparallelogram $\{-\tau\preceq u\preceq\tau\}$ is all contained in the region $\{u\in\doubleR^n \ : \ S\nabla f(u)\preceq0\}$ for some $S\in\mathcal{M}_n$, then $g(u)=f(S(u+\tau))$ for all $u\in Q(0)$ and $$ \lambda^\#=-\nabla g(0)=-S\nabla f(S(u+\tau))^T$$ defines a vector of Lagrange multipliers for (\ref{eq:01}). The proof goes exactly as in Theorem \ref{thm2}.
\end{oss}

\subsection{Decoupling the variables}\label{subsec:32}
In this subsection, we focus on the situation in which $A^TA$ is a diagonal matrix. Since:
\[
	A^TA=\begin{pmatrix}
		\norm{a_{\ast,1}}_2^2 & \langle a_{\ast,1},a_{\ast ,2}\rangle & \ldots & \langle a_{\ast, 1},a_{\ast,n}\rangle\\
		\langle a_{\ast ,2},a_{\ast,1}\rangle & \norm{a_{\ast, 2}}_2^2 & \ldots & \langle a_{\ast, 2},a_{\ast, n}\rangle \\
		\vdots & \vdots & \ddots & \vdots\\
		\langle a_{\ast ,n},a_{\ast ,1}\rangle & \langle a_{\ast ,n},a_{\ast, 2}\rangle & \ldots & \norm{a_{\ast, n}}_2^2
	\end{pmatrix}
\]
and the rank of $A^TA$ is equal to that of $A$, it follows that in this case:
\begin{equation}\label{diagA}
	A^TA=\diag(\norm{a_{\ast ,1}}_2^2,...,\norm{a_{\ast, n}}_2^2).
\end{equation}
\begin{oss}
If $m\leq n$ and $A^TA$ is diagonal, $n-m$ of the norms in (\ref{diagA}) above vanish. In this case, we assume that $a_{\ast ,m+1}=...=a_{\ast, n}=0$, so that $A$ can be written in terms of its columns as:
\[
	A=\begin{pmatrix}A' | 0_{m\times (n-m)}\end{pmatrix},
\]
where $A'=(a_{\ast, 1}|...|a_{\ast, m})\in GL(m,\doubleR)$. Observe that:
\[
	\norm{Ax-b}_2^2=\norm{A'x'-b}_2^2,
\] 
where $x'=(x_1,...,x_m)^T$, so that $x^\#$ is a mimizer of (\ref{eq:01}) if and only if $(x^\#)'=(x^\#_1,\ldots,x^\#_m)$ is a minimizer of the problem:
\begin{equation}\label{eq21capovolta}
	\text{minimize}\quad \norm{A'y-b}_2^2, y\in\doubleR^m, \ |y_j|\leq\tau_j,\ \text{$j=1,\ldots,m$},
\end{equation}
under the further condition that the remaining coordinates of $x$ vanish.

For this reason, for the rest of this subsection, we focus on (\ref{eq21capovolta}), both for the cases $n\leq m$ and $m\leq n$, and provide the Lagrange multipliers.
\end{oss}

\begin{oss} We point out that in this situation the Lagrange multipliers can be computed directly from Proposition \ref{propCaso1}. Indeed, under the \textit{orthogonality} assumption on $A$, the target function in problem (\ref{eq21capovolta}) becomes:
\[
	\sum_{j=1}^m(\norm{a_{\ast,j}}^2_2y_j^2-2y\langle a_{\ast,j},n\rangle y_j)+\norm{b}_2^2.
\]
Since the variables of all the addenda are decoupled, and the addenda are non-negative,
\[
	\min_y\sum_{j=1}^m(\norm{a_{\ast,j}}^2_2y_j^2-2y\langle a_{\ast,j},n\rangle y_j)+\norm{b}_2^2=\sum_{j=1}^m\min_{y_j}\Big(\norm{a_{\ast,j}}^2_2y_j^2-2y\langle a_{\ast,j},n\rangle y_j+\frac{\norm{b}_2^2}{m}\Big)
\] 
and a minimizer of (\ref{eq21capovolta}) is also a minimizer of the problem:
\[
	\text{minimize} \quad \norm{a_{\ast,j}}^2_2y_j^2-2y\langle a_{\ast,j},n\rangle y_j+\frac{\norm{b}_2^2}{m},\qquad |y_j|\leq \tau_j
\]
for all $j=1,...,m$. In other words, it is enough to treat (\ref{eq21capovolta}) as $m$ 1-dimensional constrained minimization problems. However, our interest is testing the tools presented in the previous section, computing the function $g$ and the separating hyperplane. 
\end{oss}

To exhibit a vector of Lagrange multipliers, we start by the set
\[
	\scriptG:=\{(u,t)\in\doubleR^{m+1} \ : \ u_j=|y_j|-\tau_j \ (j=1,\ldots,m), \ t=\norm{A'y-b}_2^2 \ \text{for some $y\in\doubleR^m$}\}.
\]
By Lemma \ref{esplG} (iii), $(u,t)\in\scriptG$ if and only if $u\succeq-\tau$ and $t=\norm{A'S(u+\tau)-b}_2^2$ for some $S\in\scriptM_m$. Let $f_S(u)=\norm{A'S(u+\tau)-b}_2^2$ and observe that:
\[
	f_S(u)=\sum_{j=1}^m\norm{a_{\ast ,j}}_2^2(u_j+\tau_j)^2-2\sum_{j=1}^ms_{jj}\langle b,a_{\ast ,j} \rangle(u_j+\tau_j)+\norm{b}_2^2.
\]

The functions $f_S$ are the equivalent of the parabolas in the $1$-dimensional case and they describe elliptic paraboloids. As it clear by Subsection \ref{subsec:31}, we need to understand what is $h_G(u):=\min_{S\in\scriptM_m}f_S(u)$. Observe that for all $S\in\scriptM_n$,
\begin{equation}\label{fundIneqSepVar}
	f_S(u)\geq\sum_{j=1}^m\norm{a_{\ast, j}}_2^2(u_j+\tau_j)^2-2\sum_{j=1}^m|\langle b,a_{\ast ,j} \rangle|(u_j+\tau_j)+\norm{b}_2^2=f_{S_\beta}(u),
\end{equation}
where $S_\beta=(s_j^\beta)_{j=1}^m\in\scriptM_m$ is a diagonal matrix such that $s_j^\beta\langle b,a_{\ast,j}\rangle \geq0 $. 

\begin{lemma}
	Under the notation and the assumptions of this subsection,
	\[
		h_G(u)=f_{S_\beta}(u)=\sum_{j=1}^m\norm{a_{\ast ,j}}_2^2(u_j+\tau_j)^2-2\sum_{j=1}^m|\langle b,a_{\ast , j} \rangle|(u_j+\tau_j)+\norm{b}_2^2.
	\]
	$h_G$ defines an elliptic paraboloid whose vertex $V=(c,0)\in\doubleR^{m+1}$ is characterized both by $c=-\tau+S_\beta (A')^{-1}b$ and
	\[
		c_j=-\tau_j+\frac{|\langle b,a_{\ast,j}\rangle|}{\norm{a_{\ast,j}}_2^2} 
	\]
	($j=1,\ldots,m$). Moreover,
	\begin{equation}\label{trdchofhG}
		h_G(u)=\sum_{j=1}^m\norm{a_{\ast,j}}_2^2(u_j-c_j)^2.
	\end{equation}
\end{lemma}
\begin{proof}
	We already proved the first part of the Lemma. We only need to compute the vertex of $f_{S_\beta}$. For, observe that the minimum of $f_{S_\beta}$ is $(c,0)$, where $c$ satisfies $f_S(c)=0$. This equation is satisfied if and only if $c=-\tau+S_\beta (A')^{-1}b$. Moreover, the minimum of $f_{S_\beta}$ is also characterized by $\nabla f_{S_\beta}(c)=0$, that is:
	\[
		c_j+\tau_j-\frac{|\langle b,a_{\ast,j}\rangle|}{\norm{a_{\ast,j}}_2^2}=0
	\]
	($j=1,\ldots,m$). Finally, using the first characterization of $c$,
	\begin{align*}
		h_G(u)&=\norm{AS_\beta(u+\tau)-b}_2^2=\norm{AS_\beta (u-c)+AS_\beta(c+\tau)-b}_2^2=\norm{AS_\beta (u-c)}_2^2=\\
		&=\sum_{j=1}^m\norm{a_{\ast,j}}_2^2(u_j-c_j)^2.
	\end{align*}
	This concludes the proof.
\end{proof}

In order to compute the Lagrange multipliers for the decoupled problem, we observe that $\scriptA+[0,+\infty)^{m+1}$ is the epigraph of the function $g(u)$ whose first properties are proved in Lemma \ref{lemmapropg}. Hence, this function describes the lower boundary of $\scriptA$, that is the part of $\scriptA$ we need to compute a separating hyperplane. By (\ref{esplg}), $g(u)=\min_{v\in Q(u)}h_G(v)$, where $Q(u)=\prod_{j=1}^m[-\tau_j,u_j]$. 

\begin{theorem}\label{thmgupart}
	Under the notation and the assumptions of this subsection, 
	\[
		g(u)=h_G(Pu),
	\]
	where $P:\prod_{j=1}^m[-\tau_j,+\infty)\to Q(c)$ is the projection defined for all $u\in\prod_{j=1}^m[-\tau_j,+\infty)$ by
	\begin{equation}\label{defPu}
		(Pu)_j=\begin{cases}
			u_j & \text{if $-\tau_j\leq u_j\leq c_j$},\\
			c_j & \text{if $u_j>c_j$}
		\end{cases}=\min\{c_j,u_j\}
	\end{equation}
	($j=1,\ldots,m$). Explicitly, under the assumptions of this subsection,
	\begin{equation}\label{explgupart}
		g(u)=\sum_{j=1}^m\norm{a_{\ast,j}}_2^2(u_j-c_j)^2\chi_{[-\tau_j,c_j]}(u_j).
	\end{equation}
	 In particular, $g\in\scriptC^1(\prod_{j=1}^n(-\tau_j,+\infty))$ with:
	\begin{equation}\label{dergpart}
		\frac{\partial g}{\partial u_j}(u)=%\frac{\partial h_G}{\partial u_j}(u)\chi_{[-\tau_j,c_j]}(u_j)=
		2\norm{a_{\ast,j}}_2^2(u_j-c_j)\chi_{[-\tau_j,c_j]}(u_j)
	\end{equation}
	for all $u\in\prod_{j=1}^n(-\tau_j,+\infty)$.
\end{theorem}
\begin{proof}
	Obviously, $P$ is a projection of $\prod_{j=1}^n[-\tau_j,+\infty)$ onto $Q(c)$. For all $j=1,\ldots,m$,
	\[
		\text{argmin}_{-\tau_j\leq v_j\leq u_j}\norm{a_{\ast,j}}_2^2(v_j-c_j)^2=\begin{cases}
			u_j & \text{if $-\tau_j\leq u_j\leq c_j$},\\
			c_j & \text{otherwise}
		\end{cases}=(Pu)_j.
	\]
	Hence,
	\[\begin{split}
		g(u)&=\min_{v\in Q(u)}h_G(v)=\sum_{j=1}^m\min_{-\tau_j\leq v_j\leq u_j}\norm{a_{\ast,j}}_2^2(v_j-c_j)^2=\\
		&=\sum_{j=1}^m\norm{a_{\ast,j}}_2^2((Pu)_j-c_j)^2=h_G(Pu).
	\end{split}
	\]
	The explicit definition of $Pu$ gives (\ref{explgupart}) and (\ref{dergpart}). The differentiability and formula (\ref{dergpart}) are obvious by the expression (\ref{explgupart}) of $g$.
\end{proof}

\begin{oss}
	As a consequence of Theorem \ref{thmgupart},
	\[
	p^\ast=g(0)=\sum_{j=1}^m\norm{a_{\ast,j}}_2^2\Big(-\tau_j+\frac{|\langle b,a_{\ast,j}\rangle|}{\norm{a_{\ast,j}}_2^2}\Big)^2\chi_{[-\tau_j,c_j]}(0).
	\]
	Then, observe that:
	\begin{equation}\label{formulachemiserve}
		-\tau_j\leq 0\leq -\tau_j+\frac{|\langle b,a_{\ast,j}\rangle|}{\norm{a_{\ast,j}}_2^2} \ \ \ \Longleftrightarrow \ \ \ 0\leq \tau_j\leq \frac{|\langle b,a_{\ast,j}\rangle|}{\norm{a_{\ast,j}}_2^2},
	\end{equation}
	so that:
	\[
		p^\ast=\sum_{j=1}^m\norm{a_{\ast,j}}_2^2\Big(-\tau_j+\frac{|\langle b,a_{\ast,j}\rangle|}{\norm{a_{\ast,j}}_2^2}\Big)^2\chi_{\Big[0,\frac{|\langle b,a_{\ast,j}\rangle|}{\norm{a_{\ast,j}}_2^2}\Big]}(\tau_j).
	\]
\end{oss}

\begin{theorem}\label{thmOptParpart}
	Under the notation of this subsection, the vector $\lambda^\#\in[0,+\infty)^m$ given by
	\[
		\lambda_j^\#=2\norm{a_{\ast,j}}_2^2\Big(\frac{|\langle b,a_{\ast,j}\rangle|}{\norm{a_{\ast,j}}_2^2}-\tau_j\Big)^+
	\]
	defines a vector of Lagrange multipliers for (\ref{eq21capovolta}).
\end{theorem}
\begin{proof}
	We apply (\ref{dergpart}) to $u=0$ and use (\ref{formulachemiserve}). Namely,
	\[
		t=p^\ast+\langle\nabla g(0),u\rangle
	\]
	is the tangent hyperplane of $g$ in $u=0$, which is also the hyperplane that separates $\scriptA$ and $\scriptB$. The direction of this hyperplane is $(\nabla g(0),-1)$, so that: 
	\[
		\lambda^\#=-\nabla g(0),
	\]
	i.e. the assertion.
\end{proof}

\begin{oss}
As far as the original problem (\ref{eq:01}) with $m\leq n$ is concerned, we get the Lagrange multipliers for free by Theorem \ref{thmOptParpart} simply observing that if $A=(a_{\ast 1}|\ldots|a_{\ast m}|0|\ldots|0)\in\doubleR^{m\times n}$,  $A'=(a_{\ast 1}|\ldots|a_{\ast m})$ and $x=(x',x'')\in\doubleR^m\times\doubleR^{n-m}$, then
\[\begin{split}
	\min_{x\in\doubleR^n}\norm{Ax-b}_2^2+\sum_{j=1}^n\lambda_j(|x_j|-\tau_j)=&\min_{x'\in\doubleR^m}\norm{A'x'-b}_2^2+\sum_{j=1}^m\lambda_j(|x'_j|-\tau_j)+\\
	&+\underbrace{\min_{x''\in\doubleR^{n-m}}\sum_{j=m+1}^n\lambda_j(|x_j|-\tau_j)}_\text{$=-\sum_{j=m+1}^n\lambda_j\tau_j$},
\end{split}
\]
so that, if $\lambda^\#\in\doubleR^m$ defines a vector of Lagrange multipliers for (\ref{eq21capovolta}), then $(\lambda^\#|0)\in\doubleR^m\times\doubleR^{n-m}$ defines a vector of Lagrange multipliers for (\ref{eq:01}).
\end{oss}

\subsection{Explicit solution}\label{subsec:es}
	The conditions $|x_j|\leq\tau_j$ are equivalent to $x_j^2\leq\tau_j^2$. Under this point of view, (\ref{eq:01}) can be restated as:
	\begin{equation}\label{Tikho}
			\text{minimize} \quad \norm{Ax-b}_2^2,\qquad x_j^2\leq\tau_j^2, 
	\end{equation}
	that can be interpreted as a \textcolor{black}{weighted} Tikhonov problem. Assume that $\lambda^\#$ is a vector of Lagrange multipliers for (\ref{eq:01}) or, equivalently, for (\ref{Tikho}). We are interested in computing
	\[
		x^\#=\arg\min_xL(x,\lambda),
	\]
	where $L$ is the Lagrange function associated to (\ref{Tikho}), i.e.
	\[
			L(x,\lambda^\#)=\norm{Ax-b}_2^2+\sum_{j=1}^n\lambda_j^\#(x_j^2-\tau_j^2).
	\]
	Since $L\in\scriptC^\infty(\doubleR^n)$ and it is convex, they satisfy $\nabla L(x,\lambda^\#)=0$, that is:
	\[
		(A^TA+\Delta_\lambda)x=A^Tb,
	\]
	where $\Delta_\lambda=\diag(\lambda_1^\#,...,\lambda_n^\#)$. Hence, $x^\#$ satisfies:
	\begin{equation}\label{xsharp}
		(A^TA+\Delta_\lambda)x^\#=A^Tb, 
	\end{equation}
	that is, $x^\#\in (A^TA+\Delta_\lambda)^{-1}A^Tb$. 
	
	\begin{oss}
		Another way to compute the Lagrange multipliers associated to (\ref{eq:01}), or equivalently to (\ref{Tikho}), can be by means of strong duality condition, namely using:
		\[
			\lambda^\#=\arg\max_{\lambda\succeq0}\min_x L(x,\lambda).
		\] 
		However, we stress that the explicit value of $\min_x L(x,\lambda)$ is still hard to compute since the implicit relation (\ref{xsharp}) satisfied by $x^\#$ cannot be made explicit by means of Dini's theorem.
	\end{oss}
 
\section{Considerations and conclusions}\label{sec:conclusions}
\subsection{Applications} Despite the apparently heavy assumptions on $A$, Theorem \ref{thmOptParpart} has itself interesting applications. For instance, it can be applied to denoising problems, where $A=I_{n\times n}$, i.e. problems in the form:
		\begin{equation}\label{denoising}
			\text{minimize}\quad \norm{x-b}_2^2, \qquad x\in\doubleR^n, \ |x_j|\leq\tau_j, \ \text{$j=1,\ldots,n$}.
	\end{equation}
	By Theorem \ref{thmOptParpart}, $\lambda^\#=(\lambda^\#_j)_{j=1}^n$ is a vector of Lagrange multipliers for (\ref{denoising}), where:
	\begin{equation}\label{den1}
		\lambda^\#_j=2(|b_j|-\tau_j)^+.
	\end{equation}
	
	We can also apply Theorem \ref{thmOptParpart} to the discrete Fourier transform, i.e. given a noisy fully-sampled signal $b\in\doubleC^{n}$, we want to find a vector $z\in\doubleC^{n}$ such that $\norm{\Phi z-b}_2^2$ is minimized under the constrains $|z_j|\leq \tau_{j}$, where $\Phi\in\doubleC^{n \times n}$ denotes the (complex) DFT matrix. Since $\Phi^\ast\Phi=I_{n\times n}$, we can apply Theorem \ref{thmOptParpart} to deduce that a set of Lagrange multipliers for this problem is:
\[
	\lambda_j^\#=2\Big(|\langle b,\phi_{\ast,j}\rangle|-\tau_j\Big)^+,
\]
($j=1,\ldots,n$), being $\phi_{\ast,j}$ the $j$-th column of $\Phi$. 

The question that naturally arises in the applications is whether the dependence of $\lambda_1,\ldots,\lambda_n$ on  $\tau_1,\ldots,\tau_n$ can be a critical issue in the applicability of the theory. Indeed, $\tau_1,\ldots,\tau_n$ are upper bounds for $|x_1|,...,|x_n|$ respectively, which are not available in the practice. However, whenever it is possible to estimate these local upper bounds, our result may lead to high-quality imaging perfomances. For instance, for denoising, (\ref{den1}) may be approximated by replacing $\tau_1,\ldots,\tau_n$ with the voxel values obtained by applying a Gaussian filter (or other types of filtering) to the noisy image. This opens the question of which filtering technique could lead to optimal approximations of the $\tau_1,\ldots,\tau_n$ depending on the field of research in which (\ref{eq:01}) can be implemented. We intend to investigate this topic in the immediate future.

\subsection{Open problems}
As long as $A^TA$ is not a diagonal matrix, the geometries of the sets $\scriptG$ and $\scriptA$ of the points satisfying (\ref{eqdefG}) and (\ref{eqdefA}) respectively, become more involved, along with the possible casuistry. However, the general case in which $A^TA$ is not diagonal would be of great importance in applications. Indeed, as long as Lagrange multipliers are proved to act as effective tuning parameters, the behavior of Lagrange multipliers for the \textcolor{black}{weighted} LASSO problem (\ref{eq:11}) in terms of voxel-wise estimates would provide a way to control the tuning parameters via estimates of the $\tau_j$. \\

Another open problem is whether it is possible to apply the same procedure to compute the Lagrange multipliers for (\ref{lassetto}). Clearly, the corresponding sets $\mathcal{G}$ and $\scriptA$ lie in $\doubleR^2$ so that $g:\doubleR\to\doubleR$. Despite this simplifying fact, the set $\scriptG$ is characterized by:
\[
	\begin{cases}
		u = s(x)^Tx-\tau,\\
		t = \norm{Ax-b}_2^2
	\end{cases} \qquad \text{for some $x\in\doubleR^n$},
\]
where $s(x)\in\doubleR^n$ is a vector such that $\diag(s(x)_j) \in \sgn(x)$ and, in this case, $u$ and $x$ belong to different spaces and a closed form for $t=t(u)$ is even more difficult to provide. \\

Finally, we stress that it would be important to generalize (\ref{eq:01}) up to consider different inner products on $\doubleR^n$. Namely, this is the situation that occurs in MRI when the undersampling pattern is non-cartesian. Problem (\ref{eq:11}) in this case becomes:
\[
\min_x\norm{Ax-b}_W^2+\sum_j\lambda_j|x_j|,
\]
where
\[
\norm{x}_W=x^TW^TWx \qquad (x\in\doubleR^n),
\]
for a definite positive diagonal matrix $W$. Since this topic falls beyond the purpose of this work, we limit ourselves to mention the very mathematical reason why the weighted norm shall definitely replace the Euclidean norm over $\doubleR^n$ when sampling is not performed on a cartesian grid. Indeed, non-cartesian sampling patterns require appropriate discretizations of the Fourier transform integral. Roughly speaking,
\[
\hat f(\xi)\thickapprox \sum_j f(x_j)e^{-2\pi i\xi\cdot x_j}\Delta x_j=\langle f,e^{2\pi i\xi\cdot }\rangle_W,
\]
where $\Delta x_j$ is the Lebesgue measure of an adequate neighborhood of $x_j$, weighting the contribution of the sample $x_j$, and $W$ is the diagonal matrix whose entries are $\sqrt{\Delta x_j}$. The inversion formula of the Fourier transform shall be modified accordingly. For instance, if the sampling follows a spiral trajectory, $\Delta x_j$ shall be bigger the further $x_j$ is from the origin, since this value serves as an avarage of $f$ on a portion of sphere that is larger as $x_j$ is far from the origin. All the above-mentioned problems will be object of our future investigations. 

\section*{Acknowledgements}
Financial support for this work has been provided by University of Bologna, the Institute of Systems Engineering at HES-SO Valais-Wallis and Schweizerischer Nationalfonds zur Förderung der Wissenschaftlichen Forschung. We are thankful to the discussion with Prof. Nicola Arcozzi, Dr. Luca Calatroni, Dr. Fabian Pedregosa and Mr. Pasquale Sirignano, who contributed to highly improve the quality of the manuscript. We are also very grateful to Prof. Micah M. Murray for his support. We acknowledge the support of The Sense Innovation and Research Center, a joint venture of the University of Lausanne (UNIL), The Lausanne University Hospital (CHUV), and The University of Applied Sciences of Western Switzerland – Valais/Wallis (HES-SO Valais/Wallis).

\section*{Authors contributions}
B.F. conceptualised the problem. B.F. and G.G. framed the hypothesis to be proved and tested. G.G. developed the proofs under B.F.'s supervision and drafted a first version of the manuscript, and B.F. shaped the article in its final version. The idea of adding the Appendix is due to B.M., who has perfected the proof of the existence provided by G.G. and contributed to the draft of the manuscript.

\section*{Declarations} The authors have not disclosed any competing interests.

\section*{Appendix}
Since we did not find a direct proof in the existing literature, we provide a formal proof of the existence of the minimizer of the generalized LASSO problem:
\begin{equation}\label{genLASSOgen}
\arg\min_{x\in \mathds{R}^n}\norm{Ax-b}_2^2+\lambda\norm{\Phi x}_1,
\end{equation}
where $b\in\mathds{R}^m$, $A\in\mathds{R}^{m\times n}$ and $\Phi\in\mathds{R}^{N\times n}$.

We state the result in the general framework of finite-dimensional vector spaces: we denote by $X,Y,Z$ three finite-dimension real vector spaces. We denote by $\langle \cdot,\cdot\rangle_X$ an inner product on $X$ and with $\norm{\cdot}_X$ the induced norm. Analogous notation will be used for $Y$, whereas we set:
\[
\norm{z}_p=\left(\sum_{j=1}^{\text{dim}(Z)}|z_j|^p\right)^{1/p},\qquad z\in Z
\]
for $0<p\leq\infty$. Recall that $\norm{\cdot}_p$ is a Banach quasi-norm (meaning that there exists $C_p\geq1$ such that $\norm{x+y}_p\leq C_p(\norm{x}_p+\norm{y}_p)$ for all $x,y\in Z$) replaces the triangular inequality for $0<p<1$, and it is a norm for $1\leq p\leq\infty$. 

Then, for given $\lambda\geq0$, $b\in Y$, $A:X\to Y$ and $\Phi:X\to Z$ linear, we define for all $x\in X$,
\[
f(x)=\norm{Ax-b}_Y^2+\lambda\norm{\Phi x}_p.
\]

\begin{theorem*}
For all $0<p\leq\infty$, $\lambda\geq0$ there exists $x^\#\in X$ such that
\[
\inf_{x\in X}f(x)=f(x^\#).
\]
In particular, the generalized LASSO problem (\ref{genLASSOgen}) has at least one solution.
\end{theorem*}
\begin{proof}
Clearly, the function $f(x)=\norm{b}_Y^2+\lambda\norm{\Phi x}_p$ attains its minimum in $x^\#=0$. Hence, we may assume that
\[
\Image(A)\neq\{0\}.
\]
Since $Image(A)$ is a vector subspace of $Y$, for all $b\in Y$ there exists a unique $y^\#\in Image(A)$ such that $\inf_{y\in Y}\norm{y-b}_Y^2=\norm{y^\#-b}_Y^2$. By definition, $y^\#=Ax^\#$ for some $x^\#\in X$. Hence,
\[
\inf_{x\in X}\norm{Ax-b}_Y^2=\inf_{y\in Image(A)}\norm{y-b}_Y^2=\min_{y\in Y}\norm{y-b}_2^2=\norm{Ax^\#-b}_Y^2
\]
and the assertion follows also for the case in which $\lambda=0$ or $Image(B)=\{0\}$. We will thereby assume that $\Image(A)\neq\{0\}$, $\Image(B)\neq\{0\}$ and $\lambda>0$. Let $L:=\text{ker}(A)\cap\text{ker}(B)=\{x\in X : Ax=0, \Phi x=0\}$ and denote the closed ball of $X$ of center $0$ and radius $r>0$ by $B_X(0,r)=\{x\in X: \norm{x}_X\leq r\}$. The rest of the proof is devided into three graded steps.

\paragraph{Step 1.} \textbf{We prove that if $L=\{0\}$, then $\lim_{\norm{x}_X\to+\infty}f(x)=+\infty$.}

By convexity of $\norm{\cdot}_Y$, 
\[
\norm{y_1}_Y^2\leq2(\norm{y_1-y_2}_Y^2+\norm{y_2}^2_Y)
\]
for all $y_1,y_2\in Y$. Therefore,
\[
\norm{Ax-b}_2^2\geq\frac{1}{2}\norm{Ax}_Y^2-\norm{b}_Y^2,
\]
so that:
\[
\norm{Ax-b}_Y^2+\lambda\norm{\Phi x}_p\geq\frac{1}{2}\norm{Ax}_Y^2+\lambda\norm{\Phi x}_p-\norm{b}_Y^2.
\]
Let
\[
\mathds{S}_X:=\{x\in X:\norm{x}_X=1\}
\]
denote the unit sphere of $X$ and set $\eta:=\min_{x\in \mathds{S}_X}\frac{1}{2}\norm{Ax}_Y^2+\lambda\norm{\Phi x}_p$. If $\eta=0$, then,
\[
\min_{x\in \mathds{S}_X}\frac{1}{2}\norm{Ax}_Y^2+\lambda\norm{\Phi x}_p=0,
\]
together with the assumptions on $\Image(A)$, $\Image(B)$ and $\lambda$, yields to the existence of $x^\#\in \mathds{S}_X$ such that $\frac{1}{2}\norm{Ax^\#}_Y^2+\lambda\norm{\Phi x^\#}_p=0$. But $\norm{\cdot}_Y$ and $\norm{\cdot}_p$ are (quasi-)norms, so $x^\#=0\notin\mathds{S}_X$. This is a contradiction. Hence, $\eta>0$. 

Next, for $\norm{x}_X>1$, we have:
\begin{align*}
\frac{1}{2}\norm{Ax}_Y^2+\lambda\norm{\phi x}_p-\norm{b}_Y^2&=\frac{1}{2}\norm{x}_X^2\norm{A\frac{x}{\norm{x}_X}}_Y^2+\lambda\norm{x}_X\norm{\Phi\frac{x}{\norm{x}_X}}_p-\norm{b}_Y^2\\
&>\norm{x}_X\left(\frac{1}{2}\norm{A\frac{x}{\norm{x}_X}}_Y^2+\lambda\norm{B\frac{x}{\norm{x}_X}}_p\right)-\norm{b}_Y^2\\
&\geq\eta\norm{x}_X-\norm{b}_Y^2.
\end{align*}
Therefore, for all $x\in X$ such that $\norm{x}_X>1$,
\[
f(x)=\norm{Ax-b}_Y^2+\lambda\norm{\Phi x}_p>\eta\norm{x}_X+\norm{b}_Y^2
\]
and the assertion follows, since $\eta>0$ implies that the right hand-side goes to $+\infty$ as $\norm{x}_X\to+\infty$.
\paragraph{Step 2.} \textbf{We prove the assertion for $L=\{0\}$.}

 Let $m:=\inf_{x\in X}f(x)$. By Step 1, there exist $R>0$ such that $f(x)>m+1$ for $\norm{x}_X>R$. $B_X(0,R)$ is compact and convex, and $\inf_{x\in X}f(x)=\inf_{x\in B_X(0,R)}f(x)$ by definition of $R$. Let $(x_j)_j\subseteq B_X(0,R)$ be a minimizing sequence. By compactness, it admits a converging subsequence and, without loss of generality, we may assume that $\lim_{j\to+\infty}x_j=x^\#\in B_X(0,R)$. By continuity, $f(x^\#)=\lim_{j\to+\infty}f(x_j)=m$.

 \paragraph{Step 3.} \textbf{We prove the assertion for $L\neq\{0\}$.} 

 Recall that $X=L\oplus L^\perp$, where the orthogonality is defined with respect to the inner product $\langle \cdot,\cdot\rangle_X$. By definition of direct sum, for all $x\in X$ there exist unique $x_1\in L$ and $x_2\in L^\perp$ such that $x=x_1+x_2$. Observe that since $x_1\in L$,
 \[
 f(x)=\norm{Ax_2-b}_Y^2+\lambda\norm{\Phi x_2}_p=f(x_2).
 \]
%in particular, $f(x)=f(Px)$, where $P:X\to L^\perp$ is the projection onto $L^\perp$, i.e. $x_2=Px$, and
In particular,
\[
\inf_{x\in X}f(x)=\inf_{x\in L^\perp}f(x).
\]
The restrictions of $A$ and $\Phi$ to $L^\perp$ are linear mappings between vector spaces. We denote them with $A|_{L^\perp}L^\perp\to Y$ and $\Phi|_{L^\perp}:L^\perp\to Z$ respectively and set $f|_{L^\perp}:L^\perp\to Y$ as the restriction of $f$ to $L^\perp$. Obviously,
\[
f|_{L^\perp}(x)=\norm{A|_{L^\perp}x-b}_Y^2+\lambda\norm{\Phi|_{L^\perp}x}_p=f(x)
\]
for all $x\in L^\perp$, so that:
\[
\inf_{x\in X}f(x)=\inf_{x\in L^\perp}f(x)=\inf_{x\in L^\perp}f|_{L^\perp}(x).
\]
Obviously,  
\[
    L^\perp:=\ker(A|_{L^\perp})\cap\ker(\Phi|_{L^\perp})=\ker(A)\cap\ker(B)\cap L^\perp =L\cap L^\perp=\{0\}.
\]
Therefore, by Step 2, it follows that there exists $x^\#\in L^\perp$ such that:
\[
\inf_{x\in L^\perp}f|_{L^\perp}(x)=f|_{L^\perp}(x^\#).
\]
This implies that:
\[
\inf_{x\in X}f(x)=f|_{L^\perp}(x^\#)=f(x^\#),
\]
since $x^\#\in L^\perp$.
\end{proof}

\end{document}